\RequirePackage{fix-cm}
\documentclass[smallextended]{svjour3}       
\smartqed  
\usepackage{csquotes}
\usepackage{graphicx}
\usepackage{tikz}
\usepackage{amsfonts}
\usepackage{amsmath}
\usepackage{amssymb}
\usepackage{geometry}
\usepackage{color}
\usepackage[hidelinks]{hyperref}

\usepackage{placeins}

\usepackage{cite}
\newcommand{\de}{\mathrm{d}}
\begin{document}

\title{Nonconforming approximation methods for function reconstruction on general polygonal meshes via orthogonal polynomials
}

\titlerunning{Nonconforming approximation methods for function reconstruction}        
\author{Francesco Dell'Accio \and Allal Guessab \and
        Gradimir V. Milovanovi{\'c} \and 
        Federico Nudo
}


\institute{            Francesco Dell'Accio \at
             Department of Mathematics and Computer Science, University of Calabria, Rende (CS), Italy\\ 
             \email{francesco.dellaccio@unical.it} 
         \and 
Allal Guessab \at
             Avenue Al Walae, 73000, 
Dakhla, Morocco\\  \email{guessaballal7@gmail.com} 
         \and
             Gradimir V. Milovanovi{\'c} \at
              Serbian Academy of Sciences and Arts (SASA), 11000 Belgrade, Serbia \& University of Ni\v s, Faculty of Sciences and Mathematics, 18000 Ni\v s, Serbia \\
              \email{gvm@mi.sanu.ac.rs}
         \and 
            Federico Nudo \at
             Department of Mathematics and Computer Science, University of Calabria, Rende (CS), Italy \\
              \email{federico.nudo@unical.it}
}

\date{Version: August 09, 2025}

\maketitle

\begin{abstract}
In this work, we introduce new families of nonconforming approximation methods for reconstructing functions on general polygonal meshes. These methods are defined using degrees of freedom based on weighted moments of orthogonal polynomials and can  reproduce higher-degree polynomials. 
This setting naturally arises in applications where pointwise evaluations are unavailable and only integral measurements over subdomains are accessible. We develop a unisolvence theory and derive necessary and sufficient conditions for the associated approximation spaces to be unisolvent. Specifically, it
is shown that unisolvence depends on the parity of the product of the polynomial degree~$m$ and the number of polygon edges~$N$. When this condition is not satisfied, we introduce an enrichment strategy involving an additional linear functional and a suitably designed enrichment function to ensure unisolvence. Numerical experiments confirm the accuracy of the proposed method.
\end{abstract}


\section{Introduction}
The reconstruction of unknown functions from integral data is a fundamental problem in approximation theory and numerical analysis. In many practical scenarios, direct pointwise evaluations of the target function are unavailable and only integral measurements over subdomains can be accessible. Such scenarios arise naturally in experimental physics, medical imaging and inverse problems~\cite{kak2001principles,palamodov2016reconstruction,recor1,recor2,guessab2019extended}. A prominent example is tomographic imaging, where the available data consist of line integrals of the unknown function along prescribed paths~\cite{natterer2001mathematics,palamodov2016reconstruction}. A classical and versatile strategy for such problems is to approximate the function by piecewise polynomial spaces defined over a partition of the computational domain $\Omega \subset \mathbb{R}^d$, $d \ge 1$. Local polynomial approximations are constructed on the subdomains and subsequently assembled into a global approximation. If the resulting global function is continuous across the interfaces, the scheme is called \emph{conforming}; otherwise, it is \emph{nonconforming}. Nonconforming approximation schemes often provide greater flexibility, particularly when continuity across element boundaries is not strictly required. These methods naturally fit problems where the available data consist of integral measurements rather than pointwise samples. Recently, enriched nonconforming approximation schemes have been successfully employed for function reconstruction from integral data~\cite{DellAccio2025new,dell2025truncated}. However, the low polynomial degree of standard approximation spaces may limit their effectiveness, especially when dealing with functions exhibiting sharp gradients or oscillatory behavior. A natural way to improve their approximation capabilities is to use higher-order polynomial spaces enriched with suitably chosen enrichment functions and additional degrees of freedom~\cite{nudosolo,nudosolo2,DellAccioCANWA}.

In this work, we introduce \emph{new families of nonconforming approximation methods that reproduce higher degree polynomials} on general polygonal meshes, aimed at solving function reconstruction problems from weighted integral data. The construction relies on edge degrees of freedom defined via orthogonal polynomials. We perform a detailed unisolvence analysis and show that the unisolvence of the associated approximation space depends on the parity of the product of the polynomial degree~$m$ and the number of polygon edges~$N$. When this condition is not satisfied, we propose an \emph{enrichment strategy} that enhances the approximation space with an additional linear functional and a suitably designed enrichment function to ensure unisolvence.

The paper is organized as follows. In Section~\ref{sec1}, we introduce boundary polynomial spaces defined by weighted edge moments and develop the main tools for the unisolvence analysis, including a general characterization theorem based on the parity of the product $mN$. We also distinguish between the cases of even and odd polynomial degree~$m$, and provide explicit necessary and sufficient conditions in each scenario. In Section~\ref{sec2}, we present an enrichment strategy to address cases where the parity condition is not satisfied, aiming to ensure unisolvence. In Section~\ref{sec3}, we introduce new families of nonconforming approximation methods of arbitrary order~$m$, combining the theoretical framework with explicit definitions of the degrees of freedom. Finally, in Section~\ref{sec4}, we provide numerical experiments on function reconstruction problems, which confirm the accuracy of the proposed method across different polynomial degrees.

\section{Boundary element triple based on orthogonal polynomials}
\label{sec1}

Let $\omega$ be a nonnegative, integrable function defined on an interval $[a,b]$, which may be either bounded or unbounded. We define the $n$-th moment of $\omega$ as
\begin{equation*}
    \mu_n = \int_a^b x^n \omega(x) {\de}x, \quad n \in \mathbb{N}_0=\mathbb{N}\cup\{0\}.
\end{equation*}
We assume that $\omega$ satisfies the following conditions:
\begin{itemize}
    \item the total mass is strictly positive
    \begin{equation*}
        \int_a^b \omega(x) {\de}x > 0;
    \end{equation*}
    \item for every $n \in \mathbb{N}_0$, the $n$-th moment is finite, i.e.
    \begin{equation*}
        \mu_n < \infty.
    \end{equation*}
\end{itemize}

A function $\omega$ satisfying these conditions is called a \emph{weight function}~\cite{chihara2011introduction}. A sequence of polynomials $\left\{\pi_n\right\}_{n \in \mathbb{N}_0}$ is said to be \emph{orthogonal on $[a,b]$ with respect to $\omega$} if
\begin{itemize}
    \item each $\pi_n$ is a polynomial of exact degree $n$, that is,
    \begin{equation*}
\operatorname{deg}\left(\pi_n\right)=n, \quad n\in\mathbb{N}_0; 
    \end{equation*}
    \item the weighted inner product 
\begin{equation*}
    \left\langle p, q \right\rangle_{\omega}=\int_{a}^b p(x) q(x) \omega(x) {\de}x, \quad   p,q\in\mathbb{P}([a,b])
\end{equation*}
    satisfies 
    \begin{equation*}
        \left\langle \pi_m, \pi_n \right\rangle_{\omega} = K_n \delta_{m n}, \quad  m,n\in\mathbb{N}_0,
    \end{equation*}
  where $K_n \neq 0$ and $\delta_{mn}$ denotes the Kronecker delta symbol. Here, $\mathbb{P}([a,b])$ is the space of real polynomials on $[a,b]$. 
\end{itemize}
If, in addition, $K_n=1$ for any $n\in\mathbb{N}_0$, the sequence $\left\{\pi_n\right\}_{n\in\mathbb{N}_0}$ is said to be \emph{orthonormal on $[a,b]$ with respect to $\omega$}. 

Let $K \subset \mathbb{R}^2$ be a polygon with vertices $\mathbf{v}_1, \dots, \mathbf{v}_N$ ordered counter-clockwise and let the edges be defined by
\begin{equation*}
  e_i = \left[\mathbf v_i,\mathbf v_{i+1}\right], \quad i=1,\dots,N,    
\end{equation*}
with the convention that $ \mathbf{v}_{N+1}=\mathbf v_1$. We introduce the function space
\begin{equation*}
    \mathcal{U}(\partial K)=\left\{u\in L^2(\partial K)\, :\, u_{{\mkern 1mu \vrule height 2ex\mkern2mu e_i}}\in C^0\left(e_i\right),\, i=1,\dots,N \right\}
\end{equation*}
that is, the space of square-integrable functions on $\partial K$  whose restrictions to each edge $e_i$ are continuous functions.
\begin{figure}
\centering
\includegraphics[width=\linewidth]{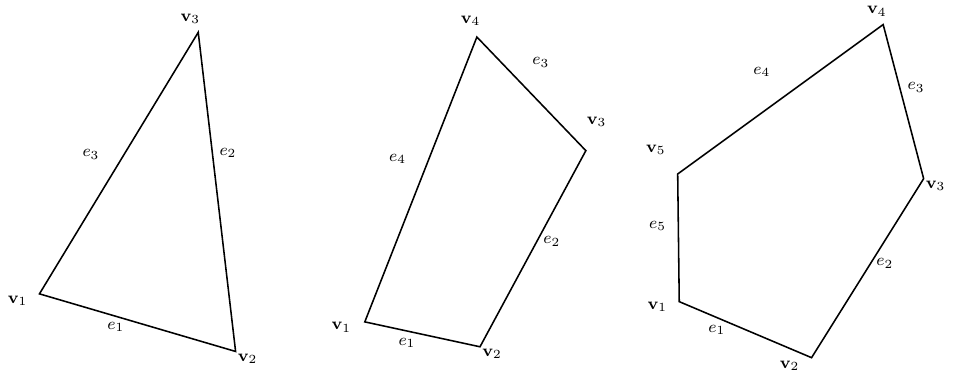}
  \caption{Plot of the polygon $K$ with $N=3$ (left), $N=4$ (center) and $N=5$ (right).}
    \label{fig:poligoni} 	
\end{figure}

For any $i = 1, \dots, N$, we define the restriction operator
\begin{equation*}
\gamma_{i}:u\in \mathcal{U}(\partial K)\rightarrow u_{|_{e_i}}\in C^0\left(e_i\right),   
\end{equation*}
which maps each function $u \in \mathcal{U}(\partial K)$ to its restriction on the edge $e_i$. For simplicity, we parametrize each edge $e_i$ using the affine map
\begin{equation*}
    \psi_i : x \in [a,b] \rightarrow \frac{b - x}{b - a} \mathbf{v}_{i} + \frac{x - a}{b - a} \mathbf{v}_{i+1}\in e_i
    \end{equation*}
and, for any $u \in \mathcal{U}(\partial K)$, we set
\begin{equation*}
    \widetilde{\gamma}_i(u)(x)=\gamma_i(u)\left(\psi_i(x)\right), \quad i = 1, \dots, N, \quad x\in[a,b].
\end{equation*}
Then, by construction
\begin{equation*}
   \widetilde{\gamma}_i(u)(a)=\gamma_i(u)\left(\mathbf{v}_{i}\right), \quad  \widetilde{\gamma}_i(u)(b)=\gamma_i(u)\left(\mathbf{v}_{i+1}\right).
   \end{equation*}

We now introduce the spaces of edgewise polynomial functions
\begin{equation*}
    \mathcal{S}_m(\partial K)=\left\{f\in L^2(\partial K) \, :\, f_{|_{e_i}}\in\mathbb{P}_m\left(e_i\right), \, i=1,\dots,N\right\},
\end{equation*}
and 
\begin{equation*}
    \mathcal{S}_m^0(\partial  K)=\mathcal{S}_m(\partial K)\cap C^0(\partial K).
\end{equation*}
In particular, $\mathcal{S}_m^0(\partial K)$ denotes the space of globally continuous functions on $\partial K$ whose restrictions to each edge are polynomials of degree at most $m$.  Here, $\mathbb{P}_m\left(e_i\right)$ denotes the space of real-valued polynomials of degree at most $m$, defined on $e_i$. Similarly, in the following, we denote by $\mathbb{P}_n([a,b])$ the space of real-valued polynomials of degree at most $n$ on the interval $[a,b]$, that is
\begin{equation*}
\mathbb{P}_n([a,b])= \left\{ p\in\mathbb{P}([a,b])\, :\, \deg(p) \leq n \right\}.    
\end{equation*}

Given a basis of the space $\mathcal{S}_{m-1}(\partial K)$
\begin{equation*}
    \mathcal{B}_m=\left\{b_j \, :\, j=1,\dots,mN\right\},
\end{equation*}
we introduce the set
\begin{equation}   \label{partialK}
    \Sigma_{\partial K}=\left\{ \mathcal{L}_j:f\in \mathcal{S}_m^0(\partial  K)\rightarrow \sum_{i=1}^N\int_{a}^b \widetilde{\gamma}_i\left(f\right)(x)\widetilde{\gamma}_i\left(b_j\right)(x)\omega(x) {\de}x \, :\, j=1,\dots,mN\right\},
\end{equation}
where $\omega$ is a weight function on $[a,b]$. Then, we define the boundary element triple as
\begin{equation}\label{triple}
\mathcal{A}_{m,N}=\left(\partial K,  \mathcal{S}_m^0(\partial  K), \Sigma_{\partial K}\right).
\end{equation}

\begin{remark}
   We observe that
\begin{equation*}
\operatorname{card}\left(\Sigma_{\partial K}\right)=\dim \left( \mathcal{S}_m^0(\partial  K)\right)=mN.
\end{equation*}
\end{remark}

In what follows, we distinguish between the general interval $[a,b]$ and the reference interval $[-1,1]$ by using the variables $x \in [a,b]$ and $t \in [-1,1]$, respectively.

\subsection{The general case} 

We present a characterization theorem that establishes a necessary and sufficient condition on the domain $K$ and the order $m$ under which the triple~\eqref{triple} defines a finite element.

\begin{theorem}\label{th1}
Let $K$ be a polygon with $N$ edges, and let $\left\{\pi_n\right\}_{n \in \mathbb{N}_0}$ be a sequence of orthogonal polynomials on $[a,b]$ with respect to a weight function $\omega$. For a fixed $m \in \mathbb{N}$, the following statements are equivalent:
    \begin{enumerate}
        \item[$1^\circ$] The polynomial $\pi_m$ satisfies
        \begin{equation}\label{cond1}
            \left( \frac{\pi_m(b)}{\pi_m(a)} \right)^N \neq 1;
        \end{equation}
        \item[$2^\circ$] The triple $\mathcal{A}_{m,N}$ defines a finite element.
    \end{enumerate}
\end{theorem}

\begin{proof}
First, we assume that $\pi_m$ satisfies the condition~\eqref{cond1} and we prove that $\mathcal{A}_{m,N}$ defines a finite element. To this end, let $p \in  \mathcal{S}_m^0(\partial  K)$ be such that
\begin{equation}\label{conds}
    \mathcal{L}_j(p) = 0, \quad j=1,\dots,mN. 
\end{equation}
We aim to show that $p=0$.

We notice that, conditions~\eqref{conds} are equivalent to requiring that
\begin{equation*}
  (\forall\, q \in \mathbb{P}_{m-1}([a,b]))\quad  \int_{a}^{b} \widetilde{\gamma}_i(p)(x)q(x) \omega(x) {\de}x = 0, \quad i = 1, \dots, N.
\end{equation*}
Therefore, for any $i = 1, \ldots, N$, there exists a constant $k_i \in \mathbb{R}$ such that
\begin{equation*}
    \widetilde{\gamma}_i(p) = k_i \pi_m.
\end{equation*}

Since $p \in  \mathcal{S}_m^0(\partial  K)$ is continuous, continuity at the shared vertex between adjacent edges $e_i$ and $e_{i+1}$ implies
\begin{equation*}
    \widetilde{\gamma}_i(p)(b) = k_i \pi_m(b) = k_{i+1} \pi_m(a) = \widetilde{\gamma}_{i+1}(p)(a),
\end{equation*}
which yields
\begin{equation*}
    k_{i+1} = \frac{\pi_m(b)}{\pi_m(a)} k_i, \quad i=1,\dots,N-1.
\end{equation*}
By iterating this relation, we obtain
\begin{equation}\label{Kn}
    k_i = \left( \frac{\pi_m(b)}{\pi_m(a)} \right)^{i-1} k_1, \quad i = 2, \dots, N.
\end{equation}

Imposing the continuity also between the edges $e_N$ and $e_1$, we have 
\begin{equation*}
\widetilde{\gamma}_N(p)(b)=    \widetilde{\gamma}_1(p)(a),
\end{equation*}
which yields
\begin{equation*}
k_N \pi_m(b)=k_1 \pi_m(a).
\end{equation*}
Thus,
\begin{equation}\label{aus111}
    k_N=\frac{\pi_m(a)}{\pi_m(b)}k_1. 
\end{equation}
On the other hand, from~\eqref{Kn}, we also have
\begin{equation}\label{ausmaaa}
    k_N=\left(\frac{\pi_m(b)}{\pi_m(a)}\right)^{N-1}k_1.
\end{equation}
Combining~\eqref{aus111} and~\eqref{ausmaaa}, we obtain 
\begin{equation*}
    k_1 \left[\left( \frac{\pi_m(b)}{\pi_m(a)} \right)^{N}-1\right]=0.
\end{equation*}  
Since $\pi_m$ satisfies~\eqref{cond1}, we conclude that $k_1 = 0$. From~\eqref{Kn}, it follows that  
\begin{equation*}
    k_2 = \cdots = k_N = 0.
\end{equation*}  
Consequently, $\widetilde{\gamma}_i(p) = 0$, for any $i = 1, \ldots, N$, which implies $p = 0$. Hence, $\mathcal{A}_{m,N}$ defines a finite element.

Conversely, we suppose by contradiction that $\pi_m$ satisfies
\begin{equation*}
    \left( \frac{\pi_m(b)}{\pi_m(a)} \right)^N = 1.
\end{equation*}
We will show that, in this case, the triple $\mathcal{A}_{m,N}$ does not define a finite element. Indeed, for any $k_1 \neq 0$, by defining $k_i$ according to~\eqref{Kn}, we can construct a nonzero function $p \in  \mathcal{S}_m^0(\partial  K)$ such that
\begin{equation*}
    \mathcal{L}_j(p) = 0, \quad j=1,\dots,mN.
\end{equation*}
Therefore, the space
\begin{equation*}
    \left\{ p \in  \mathcal{S}_m^0(\partial  K) \, : \, \mathcal{L}_j(p) = 0, \, j = 1, \ldots, mN \right\}
\end{equation*}
has dimension one. As a result, the unisolvence property fails and the triple $\mathcal{A}_{m,N}$ does not define a finite element. 
This concludes the proof.  
\end{proof}

In what follows, we present two corollaries illustrating how condition~\eqref{cond1} can be reformulated in specific settings:
\begin{itemize}
    \item first, using orthogonal polynomials defined on a symmetric interval with respect to an even weight function;
    \item second, using Jacobi polynomials.
\end{itemize}

\begin{corollary}[\textbf{The even weight}]\label{cor:symmcase}   Let $K$ be a polygon with $N$ edges, and let $\left\{\pi_n\right\}_{n \in \mathbb{N}_0}$ be a sequence of orthogonal polynomials on $[-a,a]$, $a>0,$ with respect to an even weight function $\omega$. For a fixed $m \in \mathbb{N}$, the following statements are equivalent:
\begin{enumerate}
    \item[$1^\circ$] $mN$ is odd;
    \item[$2^\circ$] The triple $\mathcal{A}_{m,N}$ defines a finite element.
\end{enumerate}
\end{corollary}

\begin{proof}
Since the interval $[-a,a]$ is symmetric with respect to zero and the weight function $\omega$ is even, each orthogonal polynomial $\pi_n$ is either even or odd, depending on whether $n$ is even or odd, respectively. Therefore,
\begin{equation*}
        \pi_n(-a)= (-1)^n\pi_n(a), \quad n\in\mathbb{N}_0.
\end{equation*}
Using this identity, we obtain
\begin{equation*}
    \left( \frac{\pi_m(a)}{\pi_m(-a)} \right)^N= (-1)^{mN}.
\end{equation*}
The result follows directly from Theorem~\ref{th1}.
\end{proof}

\begin{remark} 
Let $m$ be an even integer, and let $K$ be a polygon with $N$ edges. Let $\left\{\pi_n\right\}_{n \in \mathbb{N}_0}$ be a sequence of orthogonal polynomials on $[-a,a]$, with respect to an even weight function $\omega$. In this setting, the product $mN$ is even and, by Corollary~\ref{cor:symmcase}, the triple $\mathcal{A}_{m,N}$ does not define a finite element. 
\end{remark}

\begin{corollary}[\textbf{The Jacobi Case}]\label{cor:jacobicase}  
 Let $K$ be a polygon with $N$ edges, and let $\bigl\{\pi_n^{(\alpha,\beta)}\bigr\}_{n \in \mathbb{N}_0}$ be the sequence of Jacobi polynomials on $[-1,1]$ with parameters $\alpha, \beta > -1$, orthogonal with respect to the Jacobi weight function 
\[\omega^{(\alpha,\beta)}(t)=(1-t)^{\alpha}(1+t)^{\beta}.\]
For a fixed $m \in \mathbb{N}$, the following statements hold:
\begin{enumerate}
    \item[$1^\circ$] If $\alpha \neq \beta$, the triple $\mathcal{A}_{m,N}$ defines a finite element for all values of $m$ and $N$;
    \item[$2^\circ$] If $\alpha = \beta$, the triple $\mathcal{A}_{m,N}$ defines a finite element if and only if $mN$ is odd.
\end{enumerate}
\end{corollary}

\begin{proof}
We first consider the case $\alpha \neq \beta$. Without loss of generality, we may assume that
\begin{equation}\label{alphabeta}
\alpha > \beta.    
\end{equation}
Using the fact that the values of Jacobi polynomials at the endpoints $\pm 1$ can be expressed in terms of Gamma functions as (see~\cite{szego1975orthogonal})
\begin{equation}\label{jacobi_endpoints}
\begin{aligned}
\pi_m^{(\alpha,\beta)}(1) &= \frac{\Gamma(\alpha+m+1)}{\Gamma(\alpha+1)\Gamma(m+1)}, \\
\pi_m^{(\alpha,\beta)}(-1) &= (-1)^m \frac{\Gamma(\beta+m+1)}{\Gamma(\beta+1)\Gamma(m+1)},
\end{aligned}
\end{equation}
we obtain
\begin{equation}\label{ex11s}
    \left( \frac{\pi^{(\alpha,\beta)}_m(1)}{\pi^{(\alpha,\beta)}_m(-1)} \right)^N= (-1)^{mN} \left(\frac{\Gamma(\alpha+m+1)\Gamma(\beta+1)}{\Gamma(\alpha+1)\Gamma(\beta+m+1)}\right)^N.
\end{equation}
Since the Gamma function satisfies~\cite{Abramowitz:1948:HOM}
\begin{equation}\label{gammaprop}
    \Gamma(x+1)=x\Gamma(x), \quad x\in\mathbb{R},
\end{equation}
it follows that
\begin{equation}\label{firstcond}
\frac{\Gamma(\alpha+m+1)}{\Gamma(\alpha+1)} = (\alpha+m)\cdots (\alpha+2)
(\alpha+1),
\end{equation}
and similarly
\begin{equation}\label{seccond}
\frac{\Gamma(\beta+1)}{\Gamma(\beta+m+1)} =
\frac{1}{(\beta+m)\cdots(\beta+2)(\beta+1)}.
\end{equation}
Substituting~\eqref{firstcond} and~\eqref{seccond} into~\eqref{ex11s}, we get
\begin{equation}\label{codnsl}
    \left( \frac{\pi^{(\alpha,\beta)}_m(1)}{\pi^{(\alpha,\beta)}_m(-1)} \right)^N= (-1)^{mN} \left(\frac{(\alpha+m)\cdots (\alpha+2)
(\alpha+1)}{(\beta+m)\cdots(\beta+2)(\beta+1)}\right)^N= (-1)^{mN}\prod_{k=1}^{m}\left(\frac{\alpha + k}{\beta + k}\right)^N.
\end{equation}
By assumption~\eqref{alphabeta}, it follows that 
\begin{equation*}
    \alpha + k > \beta + k, \quad \quad k=1,\dots,m,
\end{equation*}
and hence
\begin{equation*}
\frac{\alpha + k}{\beta + k} > 1, \quad k=1,\dots,m.
\end{equation*}
Therefore,
\begin{equation*}
\prod_{k=1}^{m}\left(\frac{\alpha + k}{\beta + k}\right)^N> 1, 
\end{equation*}
since $N\in\mathbb{N}$.  

Then, by~\eqref{codnsl}, condition~\eqref{cond1} of Theorem~\ref{th1} is satisfied for all values of $m$ and $N$, and the result follows for $\alpha \neq \beta$.

Now we consider the case $\alpha=\beta$. Let $\lambda > -1/2$ be such that
\begin{equation*}
    \alpha=\beta=\lambda-\frac{1}{2}.
\end{equation*}
In this case, the Jacobi polynomials reduce to the Gegenbauer polynomials, denoted by $\bigl\{C_n^{(\lambda)}\bigr\}_{n\in\mathbb{N}_0}$. These polynomials are orthogonal on $[-1,1]$ with respect to the even weight function
\begin{equation*}
        \omega^{(\lambda)}(t)= \left(1 - t^2\right)^{\lambda - 1/2}.
\end{equation*} 
The result then follows from Corollary~\ref{cor:symmcase}.  
\end{proof}

We now present an alternative formulation of Theorem~\ref{th1} that highlights the special role played by the parity of the product $mN$.
\begin{theorem}\label{cor:gencase}  Let $K$ be a polygon with $N$ edges, and let $\left\{\pi_n\right\}_{n \in \mathbb{N}_0}$ be a sequence of orthogonal polynomials on $[a,b]$ with respect to a weight function $\omega$. For a fixed $m \in \mathbb{N}$, the following statements hold:
\begin{enumerate}
  \item[$1^\circ$] If $mN$ is odd, then the triple $\mathcal{A}_{m,N}$ defines a finite element; 
  \item [$2^\circ$] If $mN$ is even, then the following statements are equivalent: 
  \begin{enumerate}
  \item[{\rm(a)}]  The triple $\mathcal{A}_{m,N}$ defines a finite element;
  \item[{\rm(b)}]  $\left|\pi_m(b)\right|\neq \left|\pi_m(a)\right|$;
  \item[{\rm(c)}] The polynomial $\pi_m\pi_m'$ is not orthogonal to the constant functions on $[a,b]$ with respect to the constant weight function $\omega(x) = 1$.
  \end{enumerate} 
\end{enumerate}
\end{theorem}

\begin{proof}
We begin with the case where $mN$ is odd. It is well known that the zeros of the orthogonal polynomial $\pi_m$ are real, simple and lie strictly inside $(a,b)$, see~\cite[pp.~99]{mastroianni2008interpolation}. Hence we can express the polynomial $\pi_m$ as
\begin{equation*}
    \pi_m(x)=\gamma_m^{(m)}\prod_{k=1}^m \left(x-\xi_{k,m}\right), \, \text{ with } \gamma_m^{(m)} \neq 0 \text{ and } \xi_{k,m} \in (a,b), \quad k=1,\dots,m.
\end{equation*}
Evaluating at the endpoints, we find
\begin{eqnarray*}
    \pi_m(b)&=&\gamma_m^{(m)}\prod_{k=1}^m \left(b-\xi_{k,m}\right),\\ \pi_m(a)&=&\gamma_m^{(m)}\prod_{k=1}^m \left(a-\xi_{k,m}\right)=(-1)^m\gamma_m^{(m)} \prod_{k=1}^m \left(\xi_{k,m}-a\right).
\end{eqnarray*}
Therefore,
\begin{equation*}
    \left(\frac{\pi_m(b)}{\pi_m(a)}\right)^N=(-1)^{mN}C_m^N, 
\end{equation*}
where 
\begin{equation*}
    C_m^N=\prod_{k=1}^m\left(\frac{b-\xi_{k,m}}{\xi_{k,m}-a}\right)^N>0.
\end{equation*}
Since $mN$ is odd, we have 
\[
    \left( \frac{\pi_m(b)}{\pi_m(a)} \right)^N = (-1)^{mN} C_m^N=-C_m^N < 0,
\]
and thus condition~\eqref{cond1} of Theorem~\ref{th1} is satisfied. As a result $\mathcal{A}_{m,N}$ defines a finite element.

Now we consider the case where $mN$ is even. We prove the equivalence of statements (a), (b), and (c).

(a) $\Rightarrow$ (b). We assume that the triple $\mathcal{A}_{m,N}$ defines a finite element. Then, by Theorem~\ref{th1}, we have
\begin{equation}\label{cosssa}
    \left( \frac{\pi_m(b)}{\pi_m(a)} \right)^N \neq 1. 
\end{equation}
      Since $mN$ is even, either $m$ or $N$ is even. If $N$ is even, then~\eqref{cosssa} implies
    \begin{equation*}
        \frac{\pi_m(b)}{\pi_m(a)}\neq \pm 1
    \end{equation*}
and hence
    \begin{equation*}
        \pi_m(b)\neq \pm \pi_m(a).
    \end{equation*}
If instead $m$ is even, then $\pi_m(a)$ and $\pi_m(b)$ have the same sign, so
\begin{equation*}
    \pi_m(b)\neq -\pi_m(a)
\end{equation*}
and~\eqref{cosssa} implies
    \begin{equation*}
        \pi_m(b)\neq \pi_m(a).
    \end{equation*}
In both cases, we conclude that
    \begin{equation*}
        \left| \pi_m(b) \right| \neq \left| \pi_m(a) \right|.
    \end{equation*}

(b) $\Rightarrow$ (c). We suppose that $\left| \pi_m(b) \right| \neq \left| \pi_m(a) \right|$. Then, in particular
\begin{equation}\label{ssaaaa}
        \left(\pi_m(b)\right)^2\neq \left(\pi_m(a)\right)^2.
    \end{equation}
By the Fundamental Theorem of Calculus, we have
    \begin{equation} \label{condintds}
        \left(\pi_m(b)\right)^2 - \left(\pi_m(a)\right)^2 = \int_a^b \bigl[ \left(\pi_m(x)\right)^2 \bigr]^{\prime} {\de}x = 2 \int_a^b \pi_m(x) \pi_m^{\prime}(x)  {\de}x.
    \end{equation}
Combining~\eqref{ssaaaa} and~\eqref{condintds}, it follows that
    \begin{equation*}
             \int_a^b \pi_m(x) \pi_m^\prime(x) {\de}x \neq 0.
    \end{equation*}
   Therefore, the polynomial $\pi_m\pi_m'$ is not orthogonal to the constant functions on $[a,b]$, with respect to the constant weight function $\omega(x) = 1$.
   
(c) $\Rightarrow$ (a). We suppose by contradiction that $\mathcal{A}_{m,N}$ does not define a finite element. Then, again by Theorem~\ref{th1}, we have
    \begin{equation*}
        \left( \frac{\pi_m(b)}{\pi_m(a)} \right)^N = 1.
    \end{equation*}
   It follows that 
    \begin{equation*}
        \left(\pi_m(b)\right)^2 = \left(\pi_m(a)\right)^2,
    \end{equation*}
    and from~\eqref{condintds}, we deduce
    \begin{equation*}
        \int_a^b \pi_m(x) \pi_m^{\prime}(x) {\de}x = 0,
    \end{equation*}
    which implies that $\pi_m\pi_m'$ is orthogonal to the constant functions on $[a,b]$ with respect to the constant weight function $\omega(x) = 1$, in contradiction with assumption (c). This proves that $\mathcal{A}_{m,N}$ must define a finite element.

This concludes the proof.  
\end{proof}

\begin{remark}
The results above can be summarized as follows:
\begin{enumerate}
    \item For any sequence of orthogonal polynomials, the triple $\mathcal{A}_{m,N}$ defines a finite element whenever $mN$ is odd. If $mN$ is even, then $\mathcal{A}_{m,N}$ defines a finite element if and only if 
    \begin{equation*}
        \left|\pi_m(b)\right| \neq \left|\pi_m(a)\right|,
    \end{equation*}
     or, equivalently, if and only if the polynomial $\pi_m\pi'_m$ is not orthogonal to the constant functions on $[a, b]$ with respect to the constant weight function $\omega(x) = 1$.
    \item If the orthogonal polynomials are associated with an even weight function, then $\mathcal{A}_{m,N}$ defines a finite element if and only if $mN$ is odd.
    \item In the case of Jacobi polynomials with parameters $\alpha, \beta > -1$, the triple $\mathcal{A}_{m,N}$ defines a finite element for all values of $m$ and $N$ whenever $\alpha \neq \beta$. If instead $\alpha = \beta$, then $\mathcal{A}_{m,N}$ defines a finite element if and only if $mN$ is odd.
\end{enumerate}
\end{remark}

\begin{remark}
If $\left\{\pi_n\right\}_{n\in\mathbb{N}_0}$ denotes the sequence of Legendre polynomials, then in this particular case, the statements of Theorem~\ref{th1} coincide with the results presented in~\cite{bouihat2018new}. Indeed, it is well known that the Legendre polynomials satisfy the identity
\begin{equation*}
\pi_m(-t) = (-1)^m \pi_m(t), \quad t\in[-1,1],
\end{equation*}
which implies that $\pi_m$ is even when $m$ is even and odd when $m$ is odd. Therefore,
\begin{equation*}
    \left(\frac{\pi_m(b)}{\pi_m(a)}\right)^N = (-1)^{mN}
\end{equation*}
and by Theorem~\ref{th1} we have that the triple $\mathcal{A}_{m,N}$ defines a finite element if and only if $mN$ is odd.

Alternatively, this result can also be derived from Theorem~\ref{cor:gencase}. Indeed, if $mN$ is odd, then the triple defines a finite element. On the other hand, if $mN$ is even,  since $\pi'_m$ has opposite parity to $\pi_m$, the product $\pi_m\pi_m'$ is odd on the symmetric interval $[-1,1]$, and hence
\begin{equation*}
    \int_{-1}^{1} \pi_m(t)\pi_m'(t) {\de}t = 0.
\end{equation*}
Therefore, when $mN$ is even, condition~(c) of Theorem~\ref{cor:gencase} is not satisfied. Thus, in this case, the triple $\mathcal{A}_{m,N}$ does not define a finite element. 

In conclusion, Theorem~\ref{cor:gencase} recovers the result established in~\cite{bouihat2018new} for Legendre polynomials and extends it to a broader class of orthogonal polynomials. In this sense, the results presented here generalize those of~\cite{bouihat2018new}.
\end{remark}

In what follows, we distinguish between two cases depending on the parity of~$m$. To this end, we introduce the following notation, which will be used throughout the next subsections.

Let $\left\{\pi_n\right\}_{n\in\mathbb{N}_0}$ be a sequence of orthogonal polynomials on the interval $[-1,1]$ with respect to a weight function $\omega$. Given constants $a,b$ with $a < b$, we define the rescaled polynomials on $[a,b]$ by
\begin{equation}\label{scalimps}
    \pi_{n,a,b}(x) = \pi_n\left(\varphi(x)\right), \quad x \in [a,b], \quad n \in \mathbb{N}_0,
\end{equation}
where the transformation
\begin{equation}\label{rescab}
\varphi:x\in[a,b]\rightarrow \varphi(x)= \frac{a + b - 2x}{a - b}\in[-1,1], 
\end{equation}
maps the interval $[a,b]$ onto $[-1,1]$.  Let $\lambda > -1/2$ be a fixed parameter. We consider the Gegenbauer-Lobatto quadrature rule associated with the weight function
\begin{equation*}
     \omega^{(\lambda)}(\varphi(x))= \left(1 - \left(\varphi(x)\right)^2\right)^{\lambda - 1/2},
\end{equation*}
defined on the interval $[a,b]$ and given by
\begin{equation}\label{gegenlobquad}
     \int_{a}^b f(x) \omega^{(\lambda)}(\varphi(x)){\de}x \approx \omega^{(a)}f(a) + \omega^{(b)} f(b) + \sum_{i=1}^m \omega_i f\bigl(\xi_{i,a,b}^{(\lambda)}\bigr),
\end{equation}
where $\xi_{1,a,b}^{(\lambda)}, \dots, \xi_{m,a,b}^{(\lambda)}$ denote the $m$ interior nodes, which are the zeros of the derivative of the rescaled Gegenbauer polynomial  $\bigl(C_{m+1,a,b}^{(\lambda)}\bigr)^{\prime}$. As shown in~\cite{szego1975orthogonal}, the derivative of Gegenbauer polynomials satisfies
\begin{equation*}
         \bigl({C}_{m+1,a,b}^{(\lambda)}(x)\bigr)^{\prime}=\bigl({C}_{m+1}^{(\lambda)}(\varphi(x))\bigr)^{\prime}=\frac{m+2\lambda+1}{b-a}{C}_{m,a,b}^{(\lambda+1)}(x), \quad x\in[a,b],
    \end{equation*}
where we have used the identity
\begin{equation}\label{varphiprime}
    \varphi^{\prime}(x)=\frac{2}{b-a}.
\end{equation}
Therefore, $\bigl({C}_{m+1,a,b}^{(\lambda)}\bigr)^{\prime}$ and ${C}_{m,a,b}^{(\lambda+1)}$ have the same zeros. 
It is well known that the quadrature rule~\eqref{gegenlobquad} is exact for all polynomials of degree up to $2m + 1$; see~\cite{Gautschi:1997:NA}.

With this notation, we now state a lemma that will be useful in our analysis.

\begin{lemma}\label{lemmaimpsGLs}
The weights at the endpoints in the Gegenbauer-Lobatto quadrature formula~\eqref{gegenlobquad} satisfy
\begin{equation*}
    \omega^{(a)} = \omega^{(b)} = \frac{b-a}{4 \bigl(C_{m}^{(\lambda+1)}(1)\bigr)^2}\int_{-1}^1 \bigl(C_{m}^{(\lambda+1)}(t)\bigr)^2\omega^{(\lambda)}(t){\de}t.
\end{equation*}
\end{lemma}
\begin{proof}
We consider two test functions
\begin{equation*}
    f_1(x)=(x-a)\bigl({C}_{m,a,b}^{(\lambda+1)}(x)\bigr)^2, \quad f_2(x)=(b-x)\bigl({C}_{m,a,b}^{(\lambda+1)}(x)\bigr)^2, \quad x\in[a,b].
\end{equation*}
We observe that
\begin{equation*}
    \int_{a}^b \left(f_2(x)-f_1(x)\right) \omega^{(\lambda)}(\varphi(x)){\de}x=(a-b)\int_{a}^b \varphi(x)\bigl({C}_{m,a,b}^{(\lambda+1)}(x)\bigr)^2\omega^{(\lambda)}(\varphi(x)){\de}x,
\end{equation*}
where $\varphi$ is defined in~\eqref{rescab}. 
Using $\varphi$ as a change of variables with
\begin{equation*}
    \varphi^{\prime}(x)=\frac{2}{b-a},
\end{equation*}
we get
\begin{align*} 
 \int_{a}^b \left(f_2(x)-f_1(x)\right) \omega^{(\lambda)}(\varphi(x)){\de}x &=  (a-b)\int_{a}^b \varphi(x)\bigl({C}_{m,a,b}^{(\lambda+1)}(x)\bigr)^2\omega^{(\lambda)}(\varphi(x)){\de}x\\
& =-\frac{(b-a)^2}{2}\int_{a}^b \varphi(x)\bigl({C}_{m}^{(\lambda+1)}(\varphi(x))\bigr)^2\omega^{(\lambda)}(\varphi(x))\varphi^{\prime}(x){\de}x\notag\\
    &=-\frac{(b-a)^2}{2}\int_{-1}^1 t \bigl({C}_{m}^{(\lambda+1)}(t)\bigr)^2\omega^{(\lambda)}(t){\de}t=0,
\end{align*}
since the integrand is odd and the interval is symmetric with respect to zero. Hence
\begin{equation}\label{imspaus}
    \int_{a}^b \left(f_2(x)-f_1(x)\right)\omega^{(\lambda)}(\varphi(x)) {\de}x=0.
\end{equation}

On the other hand, since $f_1,f_2 \in \mathbb{P}_{2m+1}([a,b])$ and the quadrature formula~\eqref{gegenlobquad} is exact for such polynomials, we have
\begin{eqnarray*}
    \int_{a}^b f_1(x) \omega^{(\lambda)}(\varphi(x)){\de}x&=&\omega^{(b)}(b-a) \bigl({C}_{m,a,b}^{(\lambda+1)}(b)\bigr)^2,\\
        \int_{a}^b f_2(x) \omega^{(\lambda)}(\varphi(x)){\de}x&=&\omega^{(a)}(b-a) \bigl({C}_{m,a,b}^{(\lambda+1)}(a)\bigr)^2.
\end{eqnarray*}
Therefore, since
\begin{equation*}
\bigl({C}_{m,a,b}^{(\lambda+1)}(a)\bigr)^2= \bigl({C}_{m}^{(\lambda+1)}(-1)\bigr)^2=\bigl({C}_{m}^{(\lambda+1)}(1)\bigr)^2=\bigl({C}_{m,a,b}^{(\lambda+1)}(b)\bigr)^2,
\end{equation*}
by~\eqref{imspaus} and the linearity of the integral, we obtain
\begin{equation*}
     0=\int_{a}^b \bigl(f_2(x)-f_1(x)\bigr) \omega^{(\lambda)}(\varphi(x)){\de}x=\bigl(\omega^{(a)}-\omega^{(b)}\bigr)(b-a)\bigl(C_{m,a,b}^{(\lambda+1)}(b)\bigr)^2.
\end{equation*}
Hence, we conclude that $\omega^{(b)} = \omega^{(a)}$.

Now we compute the explicit value of  $\omega^{(b)}$. Using $\varphi$ as change of variables, we find
\begin{align}\notag
    \int_{a}^b f_1(x)\omega^{(\lambda)}(\varphi(x)){\de}x&= \int_{a}^b (x-a)\bigl(C_{m,a,b}^{(\lambda+1)}(x)\bigr)^2\omega^{(\lambda)}(\varphi(x)){\de}x\\\notag
   &= -\frac{1}{2}\int_{a}^b (2a-2x)\bigl(C_{m,a,b}^{(\lambda+1)}(x)\bigr)^2\omega^{(\lambda)}(\varphi(x)){\de}x\\
   \notag
   &= \frac{b-a}{2}\int_{a}^b \left(\frac{(a+b-2x)+(a-b)}{a-b}\right)\bigl(C_{m,a,b}^{(\lambda+1)}(x)\bigr)^2\omega^{(\lambda)}(\varphi(x)){\de}x\\
   \notag
   &= \frac{b-a}{2}\int_{a}^b \left(\varphi(x)+1\right)\bigl(C_{m,a,b}^{(\lambda+1)}(x)\bigr)^2\omega^{(\lambda)}(\varphi(x)){\de}x\\
   \notag&=\frac{(b-a)^2}{4} \int_{a}^b (\varphi(x)+1)\bigl(C_{m}^{(\lambda+1)}(\varphi(x))\bigr)^2\omega^{(\lambda)}(\varphi(x))\varphi^{\prime}(x){\de}x\\\notag&=\frac{(b-a)^2}{4} \int_{-1}^1 (t+1)\bigl(C_{m}^{(\lambda+1)}(t)\bigr)^2\omega^{(\lambda)}(t){\de}t\\\label{newnewnew}
   &=\frac{(b-a)^2}{4} \int_{-1}^1 \bigl(C_{m}^{(\lambda+1)}(t)\bigr)^2\omega^{(\lambda)}(t){\de}t,
\end{align}
using, as before, the fact that
\begin{equation*}
    \int_{-1}^1 t\bigl(C_{m}^{(\lambda+1)}(t)\bigr)^2\omega^{(\lambda)}(t){\de}t=0. 
\end{equation*}
On the other hand, since $f_1 \in \mathbb{P}_{2m+1}([a,b])$ and the quadrature formula~\eqref{gegenlobquad} is exact for all polynomials of degree up to $2m+1$, we have
\begin{align}\label{new4}
    \int_{a}^b f_1(x)\omega^{(\lambda)}(\varphi(x)){\de}x&=\int_{a}^b (x-a)\bigl(C_{m,a,b}^{(\lambda+1)}(x)\bigr)^2\omega^{(\lambda)}(\varphi(x)){\de}x\notag\\ 
    &=\omega^{(b)}(b-a)\bigl(C_{m}^{(\lambda+1)}(1)\bigr)^2.
\end{align}
By combining \eqref{newnewnew} and \eqref{new4}, we deduce that
\begin{equation*}
\omega^{(b)}(b-a)\bigl(C_{m}^{(\lambda+1)}(1)\bigr)^2=\frac{(b-a)^2}{4} \int_{-1}^1 \bigl(C_{m}^{(\lambda+1)}(t)\bigr)^2\omega^{(\lambda)}(t){\de}t
\end{equation*}
and then
\begin{equation*}
    \omega^{(b)}=\frac{(b-a)}{4\bigl(C_{m}^{(\lambda+1)}(1)\bigr)^2}\int_{-1}^1 \bigl(C_m^{(\lambda+1)}(t)\bigr)^2\omega^{(\lambda)}(t) {\de}t. 
\end{equation*}
This concludes the proof. 
\end{proof}

We now present a set of equivalent conditions characterizing the non-orthogonality between the derivative of a rescaled orthogonal polynomial and a weighted Gegenbauer-type function. These conditions also provide an explicit characterization of the coefficients of the polynomial $\pi_m$ when expressed in the monomial basis.

\begin{theorem}\label{cor:symmcaseintgen}
Let $\left\{\pi_{n}\right\}_{n \in \mathbb{N}_0}$ be a sequence of orthogonal polynomials on $[-1,1]$, with respect to a weight function $\omega$. For any interval $[a,b]$, let 
$\left\{\pi_{n,a,b}\right\}_{n\in\mathbb{N}_0}$, with $\pi_{n,a,b}(x)=\pi_n(\varphi(x))$, $n\in\mathbb{N}_0$, 
denote the corresponding sequence of rescaled polynomials on $[a,b]$, defined in~\eqref{scalimps} via the change of variables $\varphi$. Then, for a fixed $k\in\{0,1,\ldots,m\}$, the following statements are equivalent:
\begin{enumerate}
\item[$1^\circ$] $\pi^{(k)}_{m,a,b}(b)+(-1)^{m+k+1}\pi_{m,a,b}^{(k)}(a)\neq 0$;
\item[$2^\circ$] The polynomial ${\pi}^{(k)}_{m,a,b}(x)$ is not orthogonal to the function $\left(\varphi(x)\right)^{k+1} {C}_{m,a,b}^{(\lambda+1)}(x)$ on $[a,b]$, with respect to the weight function 
 \begin{equation*}
     \omega^{(\lambda)}(\varphi(x)) = \left(1-(\varphi(x))^2\right)^{\lambda-1/2};
 \end{equation*}
 \item[$3^\circ$] Let each polynomial $\pi_n$ be written in the monomial basis as
\begin{equation}\label{express}
        \pi_{n}(t) = \gamma_n^{(n)} t^n + \gamma_{n-1}^{(n)} t^{n-1} + \cdots + \gamma_0^{(n)} = \sum_{\iota=0}^n \gamma_{\iota}^{(n)}t^{\iota}, \quad n \in \mathbb{N}_0.
    \end{equation}
    Then
    \begin{itemize}
        \item if $m$ is odd, it holds that
        \begin{equation*}
   \sum_{\substack{\iota = k \\ \nu \ \mathrm{even}}}^{m} \frac{\nu!}{(\nu-k)!} \gamma_{\nu}^{(m)} \neq 0;
        \end{equation*}
        \item if $m$ is even, it holds that
        \begin{equation*}
         \sum_{\substack{\nu = k \\ \nu \ \mathrm{odd}}}^{m} \frac{\nu!}{(\nu-k)!} \gamma_{\nu}^{(m)} \neq 0.
        \end{equation*}
    \end{itemize}
\end{enumerate}
\end{theorem}

\begin{proof}
We prove the equivalence of the three statements.
\vspace*{2mm}

$1^\circ \Leftrightarrow 2^\circ$  For any $k=0,1,\ldots,m$, the polynomial 
\begin{equation*}
    \pi_{m,a,b}^{(k)}(x) \left(\varphi(x)\right)^{k+1} C_{m,a,b}^{(\lambda+1)}(x) \in \mathbb{P}_{2m+1}.
\end{equation*}
Since the quadrature formula~\eqref{gegenlobquad} is exact for polynomials of degree up to $2m+1$,  it follows from Lemma~\ref{lemmaimpsGLs} that
\begin{eqnarray}\label{imeps}
&&\int_{a}^b \pi^{(k)}_{m,a,b}(x) \left(\varphi(x)\right)^{k+1} {C}_{m,a,b}^{(\lambda+1)}(x)\omega^{(\lambda)}(\varphi(x)){\de}x\\ \notag
&&\qquad\quad=\omega^{(b)}\left(\pi_{m,a,b}^{(k)}(a)\left(\varphi(a)\right)^{k+1} {C}_{m}^{(\lambda+1)}\left(\varphi(a)\right)+ \pi^{(k)}_{m,a,b}(b)\left(\varphi(b)\right)^{k+1} {C}_{m}^{(\lambda+1)}\left(\varphi(b)\right)\right)\\\notag
&&\qquad\quad=\omega^{(b)}\left(\pi_{m,a,b}^{(k)}(a) (-1)^{k+1} {C}_{m}^{(\lambda+1)}(-1)+ \pi_{m,a,b}^{(k)}(b)  {C}_{m}^{(\lambda+1)}(1)\right)\\\notag
    &&\qquad\quad={C}_{m}^{(\lambda+1)}(1)\omega^{(b)}\left((-1)^{m+k+1}\pi_{m,a,b}^{(k)}(a) + \pi_{m,a,b}^{(k)}(b)  \right),
\end{eqnarray}
where in the last step we have used the parity property
\begin{equation*}
    {C}_{m}^{(\lambda+1)}(-1) = (-1)^m {C}_{m}^{(\lambda+1)}(1).
\end{equation*}
Therefore, the integral~\eqref{imeps} is nonzero if and only if
\begin{equation*}
\pi_{m,a,b}^{(k)}(b)+(-1)^{m+k+1}\pi_{m,a,b}^{(k)}(a)\neq 0.    
\end{equation*}

$1^\circ \Leftrightarrow 3^\circ$
Since $\pi_{m,a,b}(x) = \pi_m(\varphi(x))$ and $ \varphi$ is affine, the chain rule gives
\begin{equation*}
\pi_{m,a,b}^{(k)}(x) = \pi_m^{(k)}(\varphi(x)) \left (\varphi'(x)\right)^k, 
\end{equation*}
then, using~\eqref{varphiprime}, it results
\begin{equation*}
\pi_{m,a,b}^{(k)}(b)+(-1)^{m+k+1}\pi_{m,a,b}^{(k)}(a)= \left(\pi_m^{(k)}(1)+(-1)^{m+k+1}\pi_m^{(k)}(-1)  \right) \left(\frac{2}{b-a}\right)^k,
\end{equation*}
which is nonzero if and only if  
\begin{equation}\label{codaa}
\pi_m^{(k)}(1)+(-1)^{m+k+1} \pi_m^{(k)}(-1) \neq 0.
\end{equation}
Using the monomial expansion~\eqref{express}, we can write
\begin{align*}
   \pi^{(k)}_{m}(1)+(-1)^{m+k+1}\pi_{m}^{(k)}(-1) &=\sum_{\nu=k}^m \frac{\nu!}{(\nu-k)!} \gamma_{\nu}^{(m)}+\sum_{\nu=k}^m \frac{\nu!}{(\nu-k)!}\gamma_{\nu}^{(m)}(-1)^{m+\nu+1}\\
    &=\sum_{\nu=k}^m \frac{\nu!}{(\nu-k)!} \gamma_{\nu}^{(m)}\left(1+(-1)^{m+\nu+1}\right).
\end{align*}

Hence, condition~\eqref{codaa} is satisfied if and only if the following holds:
\begin{itemize}
\item If $m$ is odd, the sum of the even-index coefficients is nonzero, i.e.,
\begin{equation*}
   \sum_{\substack{\nu = k \\ \nu \ \mathrm{even}}}^{m} \frac{\nu!}{(\nu-k)!} \gamma_{\nu}^{(m)} \neq 0;
\end{equation*}
\item If $m$ is even, the sum of the odd-index coefficients is nonzero, i.e.,
\begin{equation*}
    \sum_{\substack{\nu = k \\ \nu \ \mathrm{odd}}}^{m} \frac{\nu!}{(\nu-k)!} \gamma_{\nu}^{(m)} \neq 0.
\end{equation*}
\end{itemize}

The proof is completed.  
\end{proof}

When we restrict ourselves to the case $k = m - 1$, Theorem~\ref{cor:symmcaseintgen} gives several equivalent conditions that connect the analytic and algebraic properties of orthogonal polynomials. These include relations with their derivatives, recurrence coefficients, and monomial expansions.


\begin{theorem}\label{vrthm1}
Let $\left\{\pi_{n}\right\}_{n \in \mathbb{N}_0}$ be a sequence of orthogonal polynomials on $[-1,1]$, with respect to a weight function $\omega$, expressed in the monomial basis as
\begin{equation}\label{moncodn}
        \pi_n(t) = \gamma_n^{(n)} t^n + \gamma_{n-1}^{(n)} t^{n-1} + \cdots + \gamma_0^{(n)} = \gamma_n^{(n)} \prod_{k=1}^n \left(t - \xi_{k,n}\right), \quad \gamma_n^{(n)} \neq 0, \quad n \in \mathbb{N}_0.
    \end{equation}
    Let $\left\{\widetilde{\pi}_n\right\}_{n \in \mathbb{N}_0}$ denote the corresponding sequence of monic orthogonal polynomials, defined by
    \begin{equation}\label{monicpol}
        \widetilde{\pi}_n(t) = \frac{1}{\gamma_n^{(n)}} \pi_n(t) = t^n + \widetilde{\gamma}_{n-1}^{(n)} t^{n-1} + \cdots + \widetilde{\gamma}_0^{(n)}, \quad n \in \mathbb{N}_0,
    \end{equation}
    where
    \begin{equation}\label{condortpolss}
        \widetilde{\gamma}_k^{(n)} = \frac{\gamma_k^{(n)}}{\gamma_n^{(n)}}, \quad k = 0,1, \dots, n-1.
    \end{equation}
Then, for every $m\in\mathbb{N}$, the following statements are equivalent:
    \begin{enumerate}
\item[$1^\circ$]  $\pi_{m}^{(m-1)}(1)+\pi_{m}^{(m-1)}(-1) \neq 0$;
\vspace*{2mm}
\item[$2^\circ$] The polynomial $\pi_{m}^{(m-1)}(t)$
 is not orthogonal to the function $t^{m} C_{m}^{(\lambda+1)}(t)$ on $[-1,1]$ with respect to the weight function
        \begin{equation*}
            \omega^{(\lambda)}(t) = \left(1 - t^2\right)^{\lambda - 1/2};
        \end{equation*}
\item[$3^\circ$] The coefficient $\gamma_{m-1}^{(m)} \neq 0$;
\vspace*{2mm}
\item[$4^\circ$] The following condition holds
        \begin{equation*}
            \sum_{k=1}^m \xi_{k,m} \neq 0;
        \end{equation*}
\item[$5^\circ$] The coefficients $c_k$ in the three-term recurrence relation for the monic polynomials $\left\{\widetilde{\pi}_n\right\}_{n \in \mathbb{N}_0}$
    \begin{equation}\label{threemonic}
            \widetilde{\pi}_n(t) = \left(t - c_n\right) \widetilde{\pi}_{n-1}(t) - d_n \widetilde{\pi}_{n-2}(t), \quad n \geq 1,
        \end{equation}
         with $\widetilde{\pi}_{-1}(t) = 0$, $\widetilde{\pi}_0(t) = 1$, satisfy
        \begin{equation*}
            \sum_{k=1}^m c_k \neq 0.
        \end{equation*}
    \end{enumerate}
\end{theorem}

\begin{proof}
    We prove the equivalence of the five statements.
\vspace*{2mm}

$1^\circ \Leftrightarrow 2^\circ$
      This follows from Theorem~\ref{cor:symmcaseintgen} with $a = -1$, $b = 1$ and $k = m - 1$.
\vspace*{2mm}

$1^\circ \Leftrightarrow 3^\circ$
     By differentiating~\eqref{moncodn}, we have
    \begin{align*}
    \pi_{m}^{(m-1)}(1)+\pi_{m}^{(m-1)}(-1) 
      &= \bigl(m!\gamma_m^{(m)}+(m-1)! \gamma_{m-1}^{(m)}\bigr)+\bigl(-m!\gamma_m^{(m)}+(m-1)! \gamma_{m-1}^{(m)}\bigr) \\
      &= 2(m-1)! \gamma_{m-1}^{(m)},
    \end{align*}
     which is nonzero if and only if $\gamma_{m-1}^{(m)} \neq 0$.
     
$3^\circ \Leftrightarrow 4^\circ$ Using Vieta's formulas~\cite{vinberg2003course}, which relate the coefficients of a polynomial to sums and products of its roots, we find that $\pi_m$ satisfies
\begin{equation*}
\sum_{k=1}^m \xi_{k,m}=-\frac{\gamma_{m-1}^{(m)}}{\gamma_{m}^{(m)}},
\end{equation*}
 so $\gamma_{m-1}^{(m)} \neq 0$ if and only if 
\begin{equation*}
\sum_{k=1}^m \xi_{k,m} \neq 0.
\end{equation*}

$4^\circ \Leftrightarrow 5^\circ$
 The monic polynomials satisfy the three-term recurrence relation
\begin{equation*}
\widetilde{\pi}_n(t)=\left(t-c_n\right)\widetilde{\pi}_{n-1}(t)-d_n\widetilde{\pi}_{n-2}(t), \quad n\in\mathbb{N},
\end{equation*}
with the initial conditions $\widetilde{\pi}_{-1}(t)=0$ and $\widetilde{\pi}_0(t)=1$. Comparing the coefficients of $t^{n-1}$ on both sides gives
\begin{equation*}
    c_n=\widetilde{\gamma}_{n-2}^{(n-1)}-\widetilde{\gamma}_{n-1}^{(n)}, \quad n\in\mathbb{N}.
\end{equation*}
By summing the first $k$ equalities, setting $\widetilde{\gamma}_{-1}^{(0)} = 0$, yields
\begin{equation*}
\sum_{k=1}^m c_k=\sum_{k=1}^m \left(\widetilde{\gamma}_{k-2}^{(k-1)}-\widetilde{\gamma}_{k-1}^{(k)}\right),
\end{equation*}
which simplifies to 
\begin{equation*}
   \sum_{k=1}^m c_k= -\widetilde{\gamma}_{m-1}^{(m)}.
\end{equation*}
Hence, using~\eqref{condortpolss}, we have 
\begin{equation*}
\sum_{k=1}^m c_k= -\frac{\gamma_{m-1}^{(m)}}{\gamma_{m}^{(m)}}. 
\end{equation*}
Then $\gamma_{m-1}^{(m)}\neq 0$ if and only if
\begin{equation*}
\sum_{k=1}^m c_k \neq 0.
\end{equation*}

    This concludes the proof.  
\end{proof}

\subsection{The case of even polynomial degree} 

We now focus on the case where $m$ is even, which corresponds to one of the configurations where
$mN$ is even and unisolvence may fail (see Theorem~\ref{cor:gencase}). In the following, we present a simplified version of Theorem~\ref{cor:symmcaseintgen} that holds in this specific setting.

\begin{lemma}\label{dddpart}
Let $m\in\mathbb{N}$ be an even positive integer. Let $\left\{\pi_{n}\right\}_{n \in \mathbb{N}_0}$ be a sequence of orthogonal polynomials on $[-1,1]$, with respect to a weight function $\omega$. For any interval $[a,b]$, let 
$\left\{\pi_{n,a,b}\right\}_{n\in\mathbb{N}_0}$, with $\pi_{n,a,b}(x)=\pi_n(\varphi(x))$, $n\in\mathbb{N}_0$,
denote the corresponding sequence of rescaled polynomials on $[a,b]$, defined in~\eqref{scalimps} via the change of variables $\varphi$. Then, for a fixed $k=0,1,\dots,m$, the following statements are equivalent:
\begin{enumerate}
\item[$1^\circ$] $\pi^{(k)}_{m,a,b}(b)+(-1)^{k+1} \pi_{m,a,b}^{(k)}(a)  \neq 0$;
\vspace*{1mm}

\item[$2^\circ$] The polynomial 
\begin{equation*}
    x\mapsto \left(\varphi(x)\right)^{k-1} \left( {\varphi(x) \pi}^{(k+1)}_{m,a,b}(x)-\frac{2k}{a-b}\pi^{(k)}_{m,a,b}(x)\right)
\end{equation*} 
is not orthogonal to the constant functions on $[a,b]$, with respect to the constant weight function $\omega(x)=1$;
\vspace*{1mm}

 \item[$3^\circ$] Let each polynomial $\pi_n$ be written in the monomial basis as
\begin{equation*}
        \pi_{n}(t) = \gamma_n^{(n)} t^n + \gamma_{n-1}^{(n)} t^{n-1} + \cdots + \gamma_0^{(n)} = \sum_{\iota=0}^n \gamma_{\iota}^{(n)}t^{\iota}, \quad n \in \mathbb{N}_0.
    \end{equation*}
Then, the sum of the odd-indexed coefficients in the $k$-th derivative of $\pi_m$ is nonzero, that is
        \begin{equation*}
         \sum_{\substack{\nu = k \\ \nu \ \mathrm{odd}}}^{m} \frac{\nu!}{(\nu-k)!} \gamma_{\nu}^{(m)} \neq 0.
        \end{equation*}
 \end{enumerate}
\end{lemma}

\begin{proof} 
We fix $k\in\{0,1,\ldots,m\}$. We now prove the equivalence of the three statements.
\vspace*{1mm}

$1^\circ \Leftrightarrow 2^\circ$ We consider the integral
    \begin{equation*}
       \int_{a}^b \left(\varphi(x)\right)^{k-1}\left(\varphi(x) \pi_{m,a,b}^{(k+1)}(x)-\frac{2k}{a-b}\pi_{m,a,b}^{(k)}(x)\right){\de}x.
    \end{equation*}
    By linearity, this is equal to
\begin{equation}
   \label{snsak}
   \int_{a}^b \left(\varphi(x)\right)^{k} \pi_{m,a,b}^{(k+1)}(x){\de}x-\frac{2k}{a-b}\int_{a}^b\left(\varphi(x)\right)^{k-1}\pi_{m,a,b}^{(k)}(x){\de}x.
\end{equation} 
Integrating the first term by parts, we obtain
\begin{eqnarray}\notag
&& \int_{a}^b \left(\varphi(x)\right)^{k} \pi_{m,a,b}^{(k+1)}(x){\de}x\\ 
\notag &&\qquad\qquad=\left[ \left(\varphi(x)\right)^{k} \pi_{m,a,b}^{(k)}(x)\right]_{x=a}^{x=b}-k\int_{a}^b\left(\varphi(x)\right)^{k-1}\varphi^{\prime}(x) \pi_{m,a,b}^{(k)}(x){\de}x\\\label{smal}
&&\qquad\qquad=\pi_{m,a,b}^{(k)}(b)+ (-1)^{k+1}\pi_{m,a,b}^{(k)}(a)+\frac{2k}{a-b}\int_{a}^b\left(\varphi(x)\right)^{k-1}\pi_{m,a,b}^{(k)}(x){\de}x,
\end{eqnarray} 
 where in the last step we have used that
 \begin{equation*}
     \varphi^{\prime}(x)=\frac{2}{b-a}.
 \end{equation*} 
 Substituting~\eqref{smal} in~\eqref{snsak}, we can write
 \begin{equation*}
     \int_{a}^b \left(\varphi(x)\right)^{k-1}\left(\varphi(x) \pi_{m,a,b}^{(k+1)}(x)-\frac{2k}{a-b}\pi_{m,a,b}^{(k)}(x)\right){\de}x=\pi_{m,a,b}^{(k)}(b)+ (-1)^{k+1}\pi_{m,a,b}^{(k)}(a).
 \end{equation*}
  Therefore, the integral is nonzero if and only if
\begin{equation*}
   \pi_{m,a,b}^{(k)}(b)+ (-1)^{k+1}\pi_{m,a,b}^{(k)}(a)\neq 0.
\end{equation*}
The first equivalence is proved.

Since $m$ is even, the equivalence $1^\circ \Leftrightarrow 3^\circ$ follows by Theorem~\ref{cor:symmcaseintgen}.  
\end{proof}

In the special case $k=0$,
the previous lemma provides a direct criterion for when the triple $\mathcal{A}_{m,N}$ defines a finite element. This result is especially valuable, as condition  (c) in Theorem~\ref{cor:gencase} may be challenging to verify directly, in general.

\begin{corollary}\label{ddd}
    Let $m\in\mathbb{N}$ be an even positive integer. Let $\left\{\pi_{n}\right\}_{n \in \mathbb{N}_0}$ be a sequence of orthogonal polynomials on $[-1,1]$ with respect to a weight function $\omega$. For any interval $[a,b]$, let 
$\left\{\pi_{n,a,b}\right\}_{n\in\mathbb{N}_0}$, with $\pi_{n,a,b}(x)=\pi_n(\varphi(x))$, $n\in\mathbb{N}_0$,
denote the corresponding sequence of rescaled polynomials on $[a,b]$, defined in~\eqref{scalimps} via the change of variables $\varphi$. Then, the following statements are equivalent:
    \begin{enumerate}
        \item[$1^\circ$] The triple $\mathcal{A}_{m,N}$ defines a finite element;
        \vspace*{1mm}
        \item[$2^\circ$] $ \pi_{m,a,b}(b)-\pi_{m,a,b}(a)\neq 0$;
        \vspace*{1mm}
        \item[$3^\circ$] The polynomial $\pi_{m,a,b}^{\prime}$ is not orthogonal to the constant functions on $[a,b]$, with respect to the constant weight function $\omega(x) = 1$;
        \vspace*{1mm}
        \item[$4^\circ$] The following relation holds
        \begin{equation*}
         \sum_{\substack{\nu = 0 \\ \nu \ \mathrm{odd}}}^{m}  \gamma_{\nu}^{(m)} \neq 0.
        \end{equation*}
    \end{enumerate}
\end{corollary}

\begin{proof} We prove the equivalence of the four statements.
\vspace*{1mm}

$1^\circ \Leftrightarrow 2^\circ$ We assume that $\mathcal{A}_{m,N}$ defines a finite element.  Then, by Theorem~\ref{cor:gencase}, it holds that
    \begin{equation}\label{eqsss}
            \left| \pi_{m,a,b}(b) \right| \neq \left| \pi_{m,a,b}(a) \right|.
    \end{equation}
 Since $m$ is even, the polynomial $\pi_{m,a,b}$ satisfies
\begin{equation*}
    \pi_{m,a,b}(b)\pi_{m,a,b}(a) > 0,
\end{equation*}
 so $\pi_{m,a,b}(b)$ and $\pi_{m,a,b}(a)$ have the same sign.
Thus, condition~\eqref{eqsss} is equivalent to 
\begin{equation*}
        \pi_{m,a,b}(b) \neq \pi_{m,a,b}(a).
\end{equation*}

The equivalences $2^\circ \Leftrightarrow 3^\circ \Leftrightarrow 4^\circ$ follow directly from Lemma~\ref{dddpart} in the case $k=0$.
\end{proof}

\smallskip

We now apply Corollary~\ref{ddd} to the family of Jacobi polynomials on $[-1,1]$, with $\alpha,\beta>-1$. 

\begin{example}\label{exampleJacobifirst}
Let $m \in \mathbb{N}$ be an even positive integer, and let $\bigl\{ \pi_n^{(\alpha,\beta)}\bigr\}_{n \in \mathbb{N}_0}$ denote the sequence of Jacobi polynomials on $[-1,1]$, with parameters $\alpha, \beta > -1$. For any $n \in \mathbb{N}_0$ and $t \in [-1,1]$, these polynomials satisfy the identities~\cite{szego1975orthogonal,beuchler2024recursion}
\begin{eqnarray}\label{es1}
 \pi_n^{(\alpha,\beta)}(t) &=& \frac{\Gamma(\alpha + n + 1)}{n! \Gamma(\alpha + \beta + n + 1)}
\sum_{i=0}^{n} \binom{n}{i} 
\frac{\Gamma(\alpha + \beta + n + i + 1)}{\Gamma(\alpha + i + 1)}
\left( \frac{t - 1}{2} \right)^i,\\
\pi_n^{(\alpha,\beta)}(-t)&=&(-1)^n\pi_{n}^{(\beta,\alpha)}(t). \label{es2}
\end{eqnarray}
Evaluating~\eqref{es1} at $t = 1$ and using property~\eqref{gammaprop}, we obtain
\begin{equation*}
    \pi_{m}^{(\alpha,\beta)}(1)=\frac{\Gamma(\alpha+m+1)}{m!\Gamma(\alpha+1)}=\frac{(\alpha+m)\cdots (\alpha+1)}{m!}.
\end{equation*}
By~\eqref{es2}, we also have
\begin{equation*}
    \pi_{m}^{(\alpha,\beta)}(-1)=(-1)^m\pi_{m}^{(\beta,\alpha)}(1)=\pi_{m}^{(\beta,\alpha)}(1)=\frac{\Gamma(\beta+m+1)}{m!\Gamma(\beta+1)}=\frac{(\beta+m)\cdots (\beta+1)}{m!}.
\end{equation*}
Therefore, by the Fundamental Theorem of Calculus, we get
\begin{eqnarray*}
    \int_{-1}^1 \bigl(\pi_{m}^{(\alpha,\beta)}(t)\bigr)^{\prime}{\de}t&=&\pi^{(\alpha,\beta)}_{m}(1)-\pi^{(\alpha,\beta)}_{m}(-1)\\&=&\frac{1}{m!}\left[(\alpha+m)\cdots (\alpha+1)-(\beta+m)\cdots(\beta+1) \right].
\end{eqnarray*}
Hence, we conclude that
\begin{equation*}
\int_{-1}^1 \bigl(\pi_{m}^{(\alpha,\beta)}(t)\bigr)^{\prime} {\de}t
=c_{\alpha,\beta}, 
\end{equation*}
where 
\begin{equation*}
    \begin{cases}
c_{\alpha,\beta} > 0, & \text{if }\alpha > \beta, \\
c_{\alpha,\beta} < 0, & \text{if } \alpha < \beta, \\
c_{\alpha,\beta}=0, & \text{if } \alpha = \beta.
\end{cases}
\end{equation*}
This shows, in particular, that the sign of the integral depends on the values of $\alpha$ and $\beta$. Then, by Corollary~\ref{ddd}, if $m$ is even, the triple $\mathcal{A}_{m,N}$ defines a finite element if and only if $\alpha \neq \beta$. This provides an alternative proof of Corollary~\ref{cor:jacobicase} in the case where $m$ is even. 
\end{example}

As a direct consequence of Theorem~\ref{vrthm1}, we derive a corollary for the sequence of Jacobi polynomials, yielding explicit conditions under which the triple $\mathcal{A}_{m,N}$ defines a finite element. In the following, to avoid ambiguity with the subscripts used for the parameters $\alpha$ and $\beta$, we adopt the standard notation
\begin{equation}\label{derivJac}
    \frac{\de^k}{{\de}t^k} \pi_{m}^{(\alpha,\beta)}(t)
\end{equation}
to denote the derivative of order $k$ of the Jacobi polynomial $\pi_{m}^{(\alpha,\beta)}$.

\begin{corollary}\label{vrthm}
Let $m\in\mathbb{N}$ be an even positive integer, and let $K$ be a polygon with $N$ edges. Let $\bigl\{\pi_n^{(\alpha,\beta)}\bigr\}_{n\in\mathbb{N}_0}$ be the sequence of Jacobi polynomials on $[-1,1]$ with parameters $\alpha,\beta>-1$. Using the same notation as in Theorem~\ref{vrthm1}, the following statements are equivalent:
    \begin{enumerate}
    \item[$1^\circ$] $\alpha\neq \beta$;
    \vspace*{1mm}
    \item[$2^\circ$] The triple $\mathcal{A}_{m,N}$ defines a finite element;
    \vspace*{1mm}
        \item[$3^\circ$] $\pi_m^{(\alpha,\beta)}(1)- \pi_m^{(\alpha,\beta)}(-1)\neq0$;
        \vspace*{1mm}
         \item[$4^\circ$] The coefficient $\gamma_{m-1}^{(m)} \neq 0$;
         \vspace*{1mm}
         \item[$5^\circ$] The polynomial ${\pi}_{m}^{(\alpha,\beta)}(t)$ is not orthogonal to the function $t{C}_{m}^{(\lambda+1)}(t)$ on $[-1,1]$ with respect to the weight function
        \begin{equation*}
            \omega^{(\lambda)}(t) = \left(1 - t^2\right)^{\lambda - 1/2};
        \end{equation*}
        \item[$6^\circ$] The polynomial
        \begin{equation*}
            \frac{\de^{m-1}}{{\de}t^{m-1}}\pi_{m}^{(\alpha,\beta)}(t)
        \end{equation*}
 is not orthogonal to the function $t^{m} C_{m}^{(\lambda+1)}(t)$ on $[-1,1]$ with respect to the weight function $ \omega^{(\lambda)}(t)$;
  \vspace*{1mm}      
        \item[$7^\circ$] The following condition holds
        \begin{equation*}
            \sum_{k=1}^m \xi_{k,m} \neq 0;
        \end{equation*}
        \item[$8^\circ$] The coefficients $c_k$ in the three-term recurrence relation for the monic polynomials $\bigl\{\widetilde{\pi}_n^{(\alpha,\beta)}\bigr\}_{n \in \mathbb{N}_0}$
    \begin{equation}\label{threetermrel}
            \widetilde{\pi}^{(\alpha,\beta)}_n(t) = \left(t - c_n\right) \widetilde{\pi}^{(\alpha,\beta)}_{n-1}(t) - d_n \widetilde{\pi}^{(\alpha,\beta)}_{n-2}(t), \quad n \geq 1,
        \end{equation}
         with $\widetilde{\pi}_{-1}(t) = 0$, $\widetilde{\pi}_0(t) = 1$, satisfy
        \begin{equation*}
            \sum_{k=1}^m c_k \neq 0;
        \end{equation*}
        \item[$9^\circ$] The following condition holds
      \begin{equation*}
\sum_{\substack{\nu = 0 \\ \nu \ \mathrm{odd}}}^{m} \gamma_{\nu}^{(m)} \neq 0;
\end{equation*}
\item[$10^\circ$]  The following condition holds 
\begin{equation*}
    \frac{\de^{m-1}}{{\de}t^{m-1}}\pi_{m}^{(\alpha,\beta)}(1)+ \frac{\de^{m-1}}{{\de}t^{m-1}}\pi_{m}^{(\alpha,\beta)}(-1) \neq 0.
\end{equation*}
    \end{enumerate}
\end{corollary}

\begin{proof}
    We prove the equivalence of the ten statements.
 \vspace*{1mm}

$1^\circ \Leftrightarrow 2^\circ$ Since $m$ is even, by Corollary~\ref{cor:jacobicase}, $\mathcal{A}_{m,N}$ defines a finite element if and only if $\alpha \neq \beta$.
\vspace*{1mm}

$2^\circ \Leftrightarrow 3^\circ$  By Corollary~\ref{ddd}, with $a=-1$ and $b=1$, $\mathcal{A}_{m,N}$ defines a finite element if and only if 
\begin{equation*}
    \pi_m^{(\alpha,\beta)}(1)- \pi_m^{(\alpha,\beta)}(-1)\neq0.
\end{equation*}

$1^\circ \Leftrightarrow 4^\circ$ From the monomial expansion, the coefficient $\gamma_{m-1}^{(m)}$ is given by
\begin{equation}\label{gamma111}
    \gamma_{m-1}^{(m)} = \frac{1}{(m-1)!} \frac{\de^{m-1}}{{\de}t^{m-1}} \pi_m^{(\alpha,\beta)}(t) \Big|_{t=0}.
\end{equation}
Applying the derivative identity for Jacobi polynomials (see~\cite{szego1975orthogonal}), it follows that
\begin{equation*}
    \frac{\de}{{\de}t} \pi_m^{(\alpha,\beta)}(t) = c^{(\alpha,\beta)}_1 \pi_{m-1}^{(\alpha+1,\beta+1)}(t), \quad c^{(\alpha,\beta)}_1=\frac{1}{2}\left(m+\alpha+\beta+1\right)\neq 0.
\end{equation*}
Iterating this identity, we find that for any $k = 1,\dots,m$, it results
\begin{equation}\label{gamma11}
    \frac{\de^{k}}{{\de}t^{k}} \pi_m^{(\alpha,\beta)}(t) = c^{(\alpha,\beta)}_{k} \pi_{m-k}^{(\alpha_k,\beta_k)}(t),    \quad c^{(\alpha,\beta)}_{k}=\frac{\Gamma(m+k+\alpha+\beta+1)}{2^k\Gamma(m+\alpha+\beta+1)}\neq0, 
\end{equation}
where
\begin{equation*}
    \alpha_k=\alpha+k, \quad \beta_k=\beta+k.
\end{equation*}
In particular, for $k=m-1$, we have 
\begin{equation}\label{scspc}
    \frac{\de^{m-1}}{{\de}t^{m-1}} \pi_m^{(\alpha,\beta)}(t) = c^{(\alpha,\beta)}_{m-1} \pi_{1}^{(\alpha_m,\beta_m)}(t),
\end{equation}
where
\begin{equation*}
  c^{(\alpha,\beta)}_{m-1}=\frac{\Gamma(2m+\alpha+\beta)}{2^{m-1}\Gamma(m+\alpha+\beta+1)}\neq0, \quad   \alpha_m=\alpha+m-1, \quad \beta_m=\beta+m-1. 
\end{equation*}
As shown in~\cite{szego1975orthogonal}, we get
\begin{equation}\label{gamma1}
     \pi_1^{(\alpha_m,\beta_m)}(0)=\frac{1}{2}\left(\alpha-\beta\right). 
\end{equation}
By combining equations~\eqref{gamma111}, \eqref{scspc}, and~\eqref{gamma1}, we obtain
\begin{equation*}
    \gamma_{m-1}^{(m)}=\frac{c_{m-1}^{(\alpha,\beta)} }{(m-1)!}\pi_1^{(\alpha_m,\beta_m)}(0)=\frac{c_{m-1}^{(\alpha,\beta)}}{2(m-1)!}\left(\alpha-\beta\right). 
    \end{equation*}
Therefore, $\gamma_{m-1}^{(m)} \neq 0$ if and only if $\alpha \neq \beta$. 
\vspace*{1mm}

$1^\circ \Leftrightarrow 8^\circ$ The coefficients $c_k$ in the three-term recurrence relation~\eqref{threetermrel} for monic Jacobi polynomials are given by (see~\cite{milovanovic1997orthogonal})
 \begin{equation*}
  c_1=\frac{\beta-\alpha}{2+\alpha+\beta}, \quad  c_k=\frac{(\beta-\alpha)(\beta+\alpha)}{(2k+\alpha+\beta-2)(2k+\alpha+\beta)}, \quad k\ge 2.
 \end{equation*}
By summing over $k = 1, \dots, m$, we get
 \begin{align*}
     \sum_{k=1}^m c_k&= c_1+\sum_{k=2}^m c_k= \frac{\beta-\alpha}{2+\alpha+\beta}+\sum_{k=2}^m \frac{(\beta-\alpha)(\beta+\alpha)}{(2k+\alpha+\beta-2)(2k+\alpha+\beta)}\\
     &= (\beta-\alpha)\left( \frac{1}{2+\alpha+\beta}+\sum_{k=2}^m \frac{\beta+\alpha}{(2k+\alpha+\beta-2)(2k+\alpha+\beta)}\right).
 \end{align*}
We now prove that the expression
\begin{equation}\label{coslaaap}
   \frac{1}{2+\alpha+\beta}+\sum_{k=2}^m \frac{\beta+\alpha}{(2k+\alpha+\beta-2)(2k+\alpha+\beta)}
\end{equation}
is strictly positive for any $\alpha, \beta > -1$. For any $k \geq 2$, we observe that
\begin{equation*}
\frac{1}{2k+\alpha+\beta-2}-\frac{1}{2k+\alpha+\beta}
=\frac{2}{(2k+\alpha+\beta-2)(2k+\alpha+\beta)},    
\end{equation*}
which allows us to write
\begin{equation*}
\frac{\beta+\alpha}{(2k+\alpha+\beta-2)(2k+\alpha+\beta)}
=\frac{\beta+\alpha}{2}\left(\frac{1}{2k+\alpha+\beta-2}-\frac{1}{2k+\alpha+\beta}\right).    
\end{equation*}
Then, by summing over $k = 2, \dots, m$, we find
\begin{align*}
  \sum_{k=2}^m \frac{\beta+\alpha}{(2k+\alpha+\beta-2)(2k+\alpha+\beta)}&= \frac{\beta+\alpha}{2}\sum_{k=2}^m\left(\frac{1}{2k+\alpha+\beta-2}-\frac{1}{2k+\alpha+\beta}\right)\\
&=\frac{\beta+\alpha}{2}\left(\frac{1}{2+\alpha+\beta}-\frac{1}{2m+\alpha+\beta}\right).
\end{align*}
Substituting this value in~\eqref{coslaaap}, we obtain
\begin{align*}
\frac{1}{2+\alpha+\beta}+&\sum_{k=2}^m\frac{\beta+\alpha}{(2k+\alpha+\beta-2)(2k+\alpha+\beta)}\\
&= 
\frac{1}{2+\alpha+\beta}
+\frac{\beta+\alpha}{2}\left(\frac{1}{2+\alpha+\beta}-\frac{1}{2m+\alpha+\beta}\right)\\
&= \frac{1}{2} - \frac{\beta+\alpha}{2(2m+\alpha+\beta)}\\
&=\frac{m}{2m+\alpha+\beta}>0
\end{align*}
since $m\in\mathbb{N}$ and $\alpha+\beta>-2$. Therefore, we conclude that 
\begin{equation*}
    \sum_{k=1}^m c_k\neq 0
\end{equation*}
if and only if $\alpha\neq \beta$.
\vspace*{1mm}

$3^\circ \Leftrightarrow 5^\circ$ Since $m$ is even, the equivalence follows by Theorem~\ref{cor:symmcaseintgen} with $k = 0$, $a=-1$ and $b=1$.
\vspace*{1mm}

$3^\circ \Leftrightarrow 9^\circ$ The equivalence follows by Lemma~\ref{dddpart} with $k = 0$, $a=-1$ and $b=1$.
\vspace*{1mm}

The equivalences $4^\circ \Leftrightarrow 6^\circ\Leftrightarrow 7^\circ\Leftrightarrow 10^\circ$ follow from Theorem~\ref{vrthm1}.  
\end{proof}

\begin{remark}
For any $\alpha,\beta > -1$, with  $\alpha \neq \beta$, the sequence of Jacobi polynomials $\bigl\{\pi_n^{(\alpha,\beta)}\bigr\}_{n \in \mathbb{N}_0}$ satisfies the first condition of Lemma~\ref{dddpart} for all $k = 0,1, \dots, m-1$. To see this, we fix 
$k\in\{0,1,\ldots, m-1\}$. Using the notation~\eqref{derivJac} and the general derivative formula~\eqref{gamma11}, we have
\begin{equation*}
    \frac{\de^{k}}{{\de}t^{k}} \pi_m^{(\alpha,\beta)}(t) = c^{(\alpha,\beta)}_{k} \pi_{m-k}^{(\alpha_k,\beta_k)}(t), 
\end{equation*}
where
\begin{equation*}
c_k^{(\alpha,\beta)}=\frac{\Gamma(m+k+\alpha+\beta+1)}{2^k\Gamma(m+\alpha+\beta+1)}\neq0, \quad \alpha_k=\alpha+k, \quad \beta_k=\beta+k.
\end{equation*}
Evaluating this expression at the endpoints using~\eqref{jacobi_endpoints}, we obtain
\begin{align*}
 \frac{\de^{k}}{{\de}t^{k}} \pi_m^{(\alpha,\beta)}(1)&=c_{k}^{(\alpha,\beta)}\pi_{m-k}^{(\alpha_k,\beta_k)}(1) =c_{k}^{(\alpha,\beta)} \frac{\Gamma(\alpha+m+1)}{\Gamma(\alpha+k+1)\Gamma(m-k+1)}, \\[2mm]
 \frac{\de^{k}}{{\de}t^{k}} \pi_m^{(\alpha,\beta)}(-1)&=c_{k}^{(\alpha,\beta)}
\pi_{m-k}^{(\alpha_k,\beta_k)}(-1) = c_{k}^{(\alpha,\beta)} (-1)^{m-k} \frac{\Gamma(\beta+m+1)}{\Gamma(\beta+k+1)\Gamma(m-k+1)}.
\end{align*}
Since $m$ is even, using~\eqref{gammaprop}, we obtain
\begin{align*}
    \frac{\de^{k}}{{\de}t^{k}} \pi_m^{(\alpha,\beta)}(1)+&(-1)^{k+1} \frac{\de^{k}}{{\de}t^{k}} \pi_m^{(\alpha,\beta)}(-1)\\
&=\frac{c_{k}^{(\alpha,\beta)}}{\Gamma(m-k+1)}\left(\frac{\Gamma(\alpha+m+1)}{\Gamma(\alpha+k+1)}+(-1)^{m+1}\frac{\Gamma(\beta+m+1)}{\Gamma(\beta+k+1)}\right)\\ 
&=\frac{c_k^{(\alpha,\beta)}}{\Gamma(m-k+1)}\left((\alpha+m)\cdots(\alpha+k+1)-(\beta+m)\cdots(\beta+k+1)\right)
\end{align*}
which is nonzero if and only if $\alpha \ne \beta$.
\end{remark}

\begin{remark}
Let $\left\{\pi_n\right\}_{n \in \mathbb{N}_0}$ be a sequence of orthogonal polynomials on the symmetric interval $[-a,a]$, with $a > 0$, with respect to an even weight function $\omega$. Let $\left\{\widetilde{\pi}_n\right\}_{n \in \mathbb{N}_0}$ denote the corresponding sequence of monic orthogonal polynomials as defined in~\eqref{monicpol}. In this setting, it is known that all the coefficients $c_k$ in the three-term recurrence relation~\eqref{threemonic} vanish, that is,
\[
    c_k=0, \quad k\in\mathbb{N},
\]
see~\cite{milovanovic1997orthogonal}. 
\vspace*{2mm}

The Gegenbauer polynomials, with parameter $\lambda > -1/2$, correspond to Jacobi polynomials with
\begin{equation*}
    \alpha=\beta=\lambda-\frac{1}{2},
\end{equation*}
and are orthogonal with respect to the even weight function
\begin{equation*}
    \omega^{(\lambda)}(t) = \left(1 - t^2\right)^{\lambda - 1/2}.
\end{equation*}
Therefore, in this case, if $m$ is even, condition~$8^\circ$ of Corollary~\ref{vrthm} is not satisfied, and the triple $\mathcal{A}_{m,N}$ does not define a finite element. This provides an alternative proof of Corollary~\ref{cor:symmcase} in the specific case of the Gegenbauer weight.
\end{remark}

\subsection{The case of odd polynomial degree} 

To complete our analysis, we now consider the case where $m$ is odd. In this setting, Theorem~\ref{cor:gencase} implies that if $N$ is odd, then the triple $\mathcal{A}_{m,N}$ defines a finite element. Therefore, in what follows, we assume that $N$ is even.  

As in the previous case, we present a simplified version of Theorem~\ref{cor:symmcaseintgen} that is valid in this specific setting.

\begin{lemma}\label{dddpartodd}
Let $m\in\mathbb{N}$ be an odd positive integer.  Let $\left\{\pi_{n}\right\}_{n \in \mathbb{N}_0}$ be a sequence of orthogonal polynomials on $[-1,1]$, with respect to a weight function $\omega$. For any interval $[a,b]$, let $\left\{\pi_{n,a,b}\right\}_{n\in\mathbb{N}_0}$, with 
$\pi_{n,a,b}(x)=\pi_n(\varphi(x))$, $n\in\mathbb{N}_0$, denote the corresponding sequence of rescaled polynomials on $[a,b]$, defined in~\eqref{scalimps} via the change of variables $\varphi$. Then, for a fixed $k\in\{0,1,\ldots,m\}$, the following statements are equivalent:
 \begin{enumerate}
    \item[$1^\circ$] 
$\pi^{(k)}_{m,a,b}(b)+(-1)^{k}\pi_{m,a,b}^{(k)}(a) \neq 0$;
\vspace*{2mm}

\item[$2^\circ$]The polynomial 
\[x\mapsto \left(\varphi(x)\right)^k \left( {\varphi(x) \pi}^{(k+1)}_{m,a,b}(x)-\frac{2(k+1)}{a-b}\pi^{(k)}_{m,a,b}(x)\right)\] 
is not orthogonal to the constant functions on $[a,b]$, with respect to the constant weight function $\omega(x)=1$;
\vspace*{2mm}

\item[$3^\circ$] Let each polynomial $\pi_n$ be written in the monomial basis as
\begin{equation*}
        \pi_{n}(t) = \gamma_n^{(n)} t^n + \gamma_{n-1}^{(n)} t^{n-1} + \cdots + \gamma_0^{(n)} = \sum_{\iota=0}^n \gamma_{\iota}^{(n)}t^{\iota}, \quad n \in \mathbb{N}_0.
\end{equation*}

Then, the sum of the even-indexed coefficients in the $k$-th derivative of $\pi_m$ is nonzero, that is
\[
         \sum_{\substack{\nu = k \\ \nu \ \mathrm{even}}}^{m} \frac{\nu!}{(\nu-k)!} \gamma_{\nu}^{(m)} \neq 0.
\]
\end{enumerate}
\end{lemma}

\begin{proof} 
We fix $k\in\{0,1,\ldots,m\}$ and prove the equivalence of the three statements.
\vspace*{2mm}

$1^\circ \Leftrightarrow 2^\circ$  We consider the following integral
\begin{equation*}
    \int_{a}^b \left(\varphi(x)\right)^k\left(\varphi(x) \pi_{m,a,b}^{(k+1)}(x)-\frac{2(k+1)}{a-b}\pi_{m,a,b}^{(k)}(x)\right){\de}x.
\end{equation*}
By linearity, this is equal to
\begin{equation}\label{snsak1}
    \int_{a}^b \left(\varphi(x)\right)^{k+1}\pi_{m,a,b}^{(k+1)}(x){\de}x-\frac{2(k+1)}{a-b}\int_{a}^b\left(\varphi(x)\right)^k\pi_{m,a,b}^{(k)}(x){\de}x.
\end{equation}
Integrating the first term by parts, we obtain
\begin{align}\label{smal1}
\notag
  \int_{a}^b \left(\varphi(x)\right)^{k+1}&\pi_{m,a,b}^{(k+1)}(x){\de}x  \\ \notag =& \left[ \left(\varphi(x)\right)^{k+1} \pi_{m,a,b}^{(k)}(x)\right]_{x=a}^{x=b}-(k+1)\int_{a}^b\left(\varphi(x)\right)^k\varphi^{\prime}(x) \pi_{m,a,b}^{(k)}(x){\de}x\\
    =& \pi_{m,a,b}^{(k)}(b)+(-1)^{k}\pi_{m,a,b}^{(k)}(a)+\frac{2(k+1)}{a-b}\int_{a}^b \left(\varphi(x)\right)^k \pi_{m,a,b}^{(k)}(x){\de}x,
\end{align} 
 where in the last step we have used that
 \begin{equation*}
     \varphi^{\prime}(x)=\frac{2}{b-a}.
 \end{equation*} 
 Substituting~\eqref{smal1} in~\eqref{snsak1}, we can write
 \begin{equation*}
    \int_{a}^b \left(\varphi(x)\right)^k\left(\varphi(x) \pi_{m,a,b}^{(k+1)}(x)-\frac{2(k+1)}{a-b}\pi_{m,a,b}^{(k)}(x)\right){\de}x= \pi_{m,a,b}^{(k)}(b)+(-1)^{k}\pi_{m,a,b}^{(k)}(a).
\end{equation*}
 Therefore, the integral is nonzero if and only if 
\begin{equation*}
    \pi_{m,a,b}^{(k)}(b)+(-1)^{k}\pi_{m,a,b}^{(k)}(a)\neq 0.
\end{equation*}
The first equivalence is proved. 

Since $m$ is odd, the equivalence $(1)\Leftrightarrow(3)$ follows by Theorem~\ref{cor:symmcaseintgen}. 
\end{proof}

In the special case $k=0$, the previous lemma provides a direct criterion for when the triple $\mathcal{A}_{m,N}$ defines a finite element. This result is especially valuable, as condition (c) in Theorem~\ref{cor:gencase} may be challenging to verify directly, in general.

\begin{corollary}\label{ddd1}
   Let $m\in\mathbb{N}$ be an odd positive integer.  Let $\left\{\pi_{n}\right\}_{n \in \mathbb{N}_0}$ be a sequence of orthogonal polynomials on $[-1,1]$, with respect to a weight function $\omega$. For any interval $[a,b]$, let 
$\left\{\pi_{n,a,b}\right\}_{n\in\mathbb{N}_0}$,   with      
$\pi_{n,a,b}(x)=\pi_n(\varphi(x))$, $n\in\mathbb{N}_0$,
denote the corresponding sequence of rescaled polynomials on $[a,b]$, defined in~\eqref{scalimps} via the change of variables $\varphi$. Then, the following statements are equivalent:
\begin{enumerate}
\item[$1^\circ$] The triple $\mathcal{A}_{m,N}$ defines a finite element;
\vspace*{1mm}
\item[$2^\circ$] $\pi_{m,a,b}(b) + \pi_{m,a,b}(a)\neq 0$;
\vspace*{1mm}
\item[$3^\circ$] The polynomial 
\begin{equation*}
           \widetilde{p}_{m}(x)=\frac{a+b-2x}{a-b}
        \pi_{m,a,b}^{\prime}(x)-\frac{2}{a-b}\pi_{m,a,b}(x)
        \end{equation*}
        is not orthogonal to the constant functions on $[a,b]$, with respect to the constant weight function $\omega(x)=1$;
        \vspace*{1mm}
\item[$4^\circ$] The following relation holds
        \begin{equation*}
         \sum_{\substack{\nu = 0 \\ \nu \ \mathrm{even}}}^{m}  \gamma_{\nu}^{(m)} \neq 0.
        \end{equation*}
    \end{enumerate}
\end{corollary}

\begin{proof}
We prove the equivalence of the four statements.
\vspace*{2mm}

$1^\circ \Leftrightarrow 2^\circ$
    We assume that $\mathcal{A}_{m,N}$ defines a finite element.  Then, by Theorem~\ref{cor:gencase}, it holds that
    \begin{equation}\label{eqsss1}
            \left| \pi_{m,a,b}(b) \right| \neq \left| \pi_{m,a,b}(a) \right|.
    \end{equation}
 Since $m$ is odd, the polynomial $\pi_{m,a,b}$ satisfies
\begin{equation*}
    \pi_{m,a,b}(b)\pi_{m,a,b}(a) < 0,
\end{equation*}
 so $\pi_{m,a,b}(b)$ and $\pi_{m,a,b}(a)$ have the opposite sign.
Thus, condition~\eqref{eqsss1} is equivalent to
\begin{equation*}
        \pi_{m,a,b}(b) \neq -\pi_{m,a,b}(a).
\end{equation*}

The equivalences $2^\circ \Leftrightarrow 3^\circ\Leftrightarrow 4^\circ$  follow directly from Lemma~\ref{dddpartodd} in the case $k=0$. 
\end{proof}

\begin{remark}
For any parameters $\alpha,\beta > -1$, such that $\alpha \neq \beta$, the sequence of Jacobi polynomials $\bigl\{\pi_n^{(\alpha,\beta)}\bigr\}_{n \in \mathbb{N}_0}$ satisfies the first condition of Lemma~\ref{dddpartodd} for all $k = 0, 1,\ldots, m-1$. To see this, we fix $k \in \{0,1, \ldots, m-1\}$. Using the notation~\eqref{derivJac} and the general derivative formula~\eqref{gamma11}, we have
\begin{equation*}
    \frac{\de^{k}}{{\de}t^{k}} \pi_m^{(\alpha,\beta)}(t) = c^{(\alpha,\beta)}_{k} \pi_{m-k}^{(\alpha_k,\beta_k)}(t), 
\end{equation*}
where
\begin{equation*}
 c_k^{(\alpha,\beta)}=\frac{\Gamma(m+k+\alpha+\beta+1)}{2^k\Gamma(m+\alpha+\beta+1)}\neq0, \quad \alpha_k=\alpha+k, \quad \beta_k=\beta+k.
\end{equation*}
Evaluating this expression at the endpoints using~\eqref{jacobi_endpoints}, we obtain
\begin{align*}
 \frac{\de^{k}}{{\de}t^{k}} \pi_m^{(\alpha,\beta)}(1)&=c_{k}^{(\alpha,\beta)}\pi_{m-k}^{(\alpha+k,\beta+k)}(1) =c_{k}^{(\alpha,\beta)} \frac{\Gamma(\alpha+m+1)}{\Gamma(\alpha+k+1)\Gamma(m-k+1)}, \\[2mm]
 \frac{\de^{k}}{{\de}t^{k}} \pi_m^{(\alpha,\beta)}(-1)&=c_{k}^{(\alpha,\beta)}
\pi_{m-k}^{(\alpha+k,\beta+k)}(-1) = c_{k}^{(\alpha,\beta)} (-1)^{m-k} \frac{\Gamma(\beta+m+1)}{\Gamma(\beta+k+1)\Gamma(m-k+1)}.
\end{align*}
Since $m$ is odd, using~\eqref{gammaprop}, we obtain
\begin{align*}
   \frac{\de^{k}}{{\de}t^{k}} \pi_m^{(\alpha,\beta)}(1)+&(-1)^k \frac{d^{k}}{{\de}t^{k}} \pi_m^{(\alpha,\beta)}(-1)\\
&=\frac{c_{k}^{(\alpha,\beta)}}{\Gamma(m-k+1)}\left(\frac{\Gamma(\alpha+m+1)}{\Gamma(\alpha+k+1)}+(-1)^{m}\frac{\Gamma(\beta+m+1)}{\Gamma(\beta+k+1)}\right)\\
&=\frac{c_k^{(\alpha,\beta)}}{\Gamma(m-k+1)}\left((\alpha+m)\cdots(\alpha+k+1)-(\beta+m)\cdots(\beta+k+1)\right),
\end{align*}
which is nonzero if and only if $\alpha \ne \beta$.
\end{remark}


\section{Enrichment process}
\label{sec2}

We now consider the case where the triple $ \mathcal{A}_{m,N}$ does not define a finite element. According to Theorem~\ref{cor:gencase}, this happens if and only if $mN$ is even and
\begin{equation*}
    \left|\pi_m(b)\right|= \left|\pi_m(a)\right|.
\end{equation*}

In this case, we propose an enrichment strategy to enhance $\mathcal{A}_{m,N}$ and ensure that it defines a finite element. This technique is commonly used to improve the approximation properties of finite elements by means of enrichment functions and enriched linear functionals; see, e.g., \cite{Guessab:2017:AUA, Guessab:2022:SAB, DellAccio:2022:AUE, DellAccio2023AGC, DellAccio:2022:ESF, DellAccioCANWA, nudosolo, nudosolo2, DellAccio2025new, dell2025truncated}. In our framework, this approach is specifically designed to ensure the unisolvence of the enriched triple. To this end, without loss of generality, we assume that the basis $\mathcal{B}_m$ of the space $\mathcal{S}_{m-1}(\partial K)$, which defines the degrees of freedom of the triple $\mathcal{A}_{m,N}$, satisfies the following structural property. We adopt a two-index notation for the elements $b_j \in \mathcal{B}_m$, namely
\begin{equation*}
   b_j=b_{\mu\nu}, \quad \mu=1,\ldots,N, \quad \nu=0,1,\ldots,m-1,
\end{equation*}
such that, for any $i = 1, \ldots, N$, the following relation holds
\begin{equation*}
\widetilde{\gamma}_i\left(b_{\mu\nu}\right)(x)={\widehat{\pi}_{\nu}(x)}\delta_{i\mu}, \quad \mu=1,\ldots,N, \quad \nu=0,1,\ldots,m-1,  
\end{equation*}
where $\left\{\widehat{\pi}_n\right\}_{n\in\mathbb{N}_0}$ is the sequence of orthonormal polynomials on $[a,b]$ with respect to the weight function $\omega$. 
Accordingly, we define the inner product
\begin{equation*}
\left\langle f,g \right\rangle_{\partial K}=\sum_{i=1}^N\int_{a}^b \widetilde{\gamma}_i\left(f\right)(x)\widetilde{\gamma}_i\left(g\right)(x)\omega(x) {\de}x, \quad f,g\in  \mathcal{U}(\partial K),
\end{equation*}
under which the basis functions are orthonormal, that is,
\begin{equation}\label{ortbl}
\left\langle b_j,b_l \right\rangle_{\partial K}=\delta_{jl}, \quad j,l=1,\ldots,mN.
\end{equation}
We notice that each linear functional in $\Sigma_{\partial K}$ can equivalently be written as
\begin{equation*}
\mathcal{L}_j(p)=\left\langle p,b_j \right\rangle_{\partial K}, \quad j=1,\dots,mN.
\end{equation*}

\begin{remark}\label{coslpa1}
By definition
\begin{equation*}
b_j \notin  \mathcal{S}_m^0(\partial  K), \quad j = 1, \dots, mN,
\end{equation*}
since $\pi_n(a)\pi_n(b) \neq 0$ for any $n\in\mathbb{N}_0$.
\end{remark}

In this setting, we introduce the enriched linear functional
\begin{equation}\label{deff}
\mathcal{F}:f\in \mathcal{S}_m^0(\partial  K)\rightarrow \sum_{i=1}^{N} \left(\frac{\pi_m(b)}{\pi_m(a)}\right)^{i-1} \int_{a}^b \widetilde{\gamma}_i(f)(x) \pi_{m}(x) \omega(x) {\de}x. 
\end{equation}

To proceed with our analysis, we establish the following key lemma.
\begin{lemma}\label{fffffff}
   Let $f\in \mathcal{S}_m^0(\partial  K)$ be such that
       \begin{align} \label{newss}
        \mathcal{L}_j(f)&=0, \quad j=1,\ldots,mN, \\[1mm] 
        \mathcal{F}(f)&=0. \label{fcond}
    \end{align}
    Then $f=0$.
\end{lemma}

\begin{proof}
As in the proof of Theorem~\ref{th1}, any $f \in  \mathcal{S}_m^0(\partial  K)$ satisfying~\eqref{newss} must be of the form
\begin{equation*}
\widetilde{\gamma}_i(f)(x)=k_i\pi_m(x)=\left(\frac{\pi_m(b)}{\pi_m(a)}\right)^{i-1}k_1 \pi_m(x), \quad x\in[a,b], \quad i=1,\ldots,N, 
\end{equation*}
where $k_1\in\mathbb{R}$.
By applying condition~\eqref{fcond}, we obtain
\begin{equation*}
    0=\mathcal{F}(f)=  \sum_{i=1}^{N} \left(\frac{\pi_m(b)}{\pi_m(a)}\right)^{2(i-1)}k_1 \int_{a}^b \left(\pi_m(x)\right)^2  \omega(x) {\de}x= k_1 \left\lVert \pi_m \right\rVert_{2,\omega}^2\sum_{i=1}^{N} \left(\frac{\pi_m(b)}{\pi_m(a)}\right)^{2(i-1)},
\end{equation*}
where 
\begin{equation*}
    \left\lVert \pi_m\right\rVert_{2,\omega}^2=\left\langle \pi_m,\pi_m\right\rangle_{\omega}=\int_{a}^b \left(\pi_m(x)\right)^2 \omega(x) {\de}x
\end{equation*}
is the weighted $L^2$-norm. Since
\begin{equation*}
    \left\lVert \pi_m \right\rVert_{2,\omega}^2 > 0 \quad \text{and} \quad \sum_{i=1}^{N} \left( \frac{\pi_m(b)}{\pi_m(a)} \right)^{2(i-1)} > 0,
\end{equation*}
it follows that $k_1 = 0$. Therefore $f = 0$.  
\end{proof}

We now fix a basis of $ \mathcal{S}_m^0(\partial  K)$, denoted by
\begin{equation*}
\mathcal{B}^{\prime}_{m}=\left\{q_1,\ldots,q_{mN}\right\}.
\end{equation*}
Since the triple $\mathcal{A}_{m,N}$ does not define a finite element, there exists a nonzero vector 
\begin{equation*}
\mathbf{c}=\left[c_1,\ldots,c_{mN}\right]^T \neq [0,\ldots,0]^T   
\end{equation*}
such that the function
\begin{equation*}
   p=\sum_{l=1}^{mN}c_l q_l\neq 0,
\end{equation*}
satisfies the conditions 
\begin{equation*}
    \mathcal{L}_j(p)=\left\langle p,b_j \right\rangle_{\partial K}=\sum_{i=1}^N\int_{a}^b \widetilde{\gamma}_i\left(p\right)(x)\widetilde{\gamma}_i\left(b_j\right)(x)\omega(x) {\de}x=0, \quad j=1,\dots,mN.
\end{equation*}
Therefore, $\mathbf{c}$ is a nontrivial solution of the homogeneous linear system
\begin{equation*}
    A\mathbf{c}=[0,\dots,0]^T,
\end{equation*}
where
 \begin{equation}\label{pr1}
   A=\begin{bmatrix}
\left\langle q_1,b_1 \right\rangle_{\partial K} & \left\langle q_2,b_1 \right\rangle_{\partial K} & \cdots & \left\langle q_{mN},b_1 \right\rangle_{\partial K}\\
\left\langle q_1,b_2 \right\rangle_{\partial K} & \left\langle q_2,b_2 \right\rangle_{\partial K} & \cdots & \left\langle q_{mN},b_2 \right\rangle_{\partial K}\\
\vdots  & \vdots  & \ddots & \vdots  \\
\left\langle q_1,b_{mN} \right\rangle_{\partial K} & \left\langle q_2,b_{mN} \right\rangle_{\partial K} & \cdots & \left\langle q_{mN},b_{mN} \right\rangle_{\partial K}\\
\end{bmatrix}.
\end{equation}
Consequently, the matrix $A$ is singular.
\vspace*{2mm}

The next lemma plays a central role in the enrichment process.

\begin{lemma}\label{lemmaslsa}
There exists a vector
\begin{equation*}
\boldsymbol{d} = \left[d_1,\ldots,d_{mN}\right]^T \neq [0,\ldots,0]^T
\end{equation*}
such that the linear functional
\begin{equation}\label{defG}
\mathcal{G} = \sum_{j=1}^{mN} d_j \mathcal{L}_j
\end{equation}
annihilates all functions in $ \mathcal{S}_m^0(\partial  K)$, that is,
\begin{equation*}
\mathcal{G}(p) = 0\quad \text{for any}   
\ p\in\mathcal{S}_m^0(\partial  K). 
\end{equation*}
\end{lemma}

\begin{proof}
Since $\mathcal{B}^{\prime}_{m}=\left\{q_1,\dots,q_{mN}\right\}$ is a basis of $\mathcal{S}_m^0(\partial K)$, it is sufficient to prove that 
\begin{equation*}
\mathcal{G}\left(q_l\right) = \sum_{j=1}^{mN} d_j \mathcal{L}_j\left(q_l\right) =0 , \quad l = 1, \dots, mN,
\end{equation*}
or, equivalently 
\begin{equation*}
    \sum_{j=1}^{mN} d_j \left\langle q_l,b_j\right\rangle_{\partial K}= 0, \quad l = 1, \dots, mN.
\end{equation*}
These conditions define a linear system for the components of $\boldsymbol{d}$, which can be written as
\begin{equation*}
  \left\{\begin{array}{cl}
d_1 \left\langle q_1, b_1\right\rangle_{\partial K}+\dots+d_{mN} \left\langle q_1, b_{mN}\right\rangle_{\partial K}&=0, \\[2mm]
d_1 \left\langle q_2, b_1\right\rangle_{\partial K}+\cdots+d_{mN} \left\langle q_2, b_{mN}\right\rangle_{\partial K}&=0, \\[2mm]
&\ \vdots \\
d_1 \left\langle q_{mN}, b_1\right\rangle_{\partial K}+\cdots+d_{mN} \left\langle q_{mN}, b_{mN}\right\rangle_{\partial K}&=0,
\end{array}\right.
\end{equation*}
or in matrix form as  
\begin{equation*}
    A^{T}\boldsymbol{d}=[0,\dots,0]^T, 
\end{equation*}
where $A$ is the matrix defined in~\eqref{pr1}. Since $A$ is singular, its transpose $A^T$ is also singular. It follows that the kernel of $A^T$ is nontrivial and hence there exists a nonzero vector $\boldsymbol{d} \in \ker\left(A^T\right)$, which proves the claim.  
\end{proof}

In the following, we consider the linear functional $\mathcal{G}$ defined in~\eqref{defG} and assume, without loss of generality, that
\begin{equation*}
    d_1\neq 0.
\end{equation*}
Under this assumption and in order to enrich our finite element space, we define
\begin{equation}\label{fhat}
\widetilde{f}=b_{1}.
\end{equation}
By Remark~\ref{coslpa1}, it follows that
\begin{equation*}
\widetilde{f}\notin{\mathcal{S}}^0_m(\partial K).
\end{equation*}
We then define the enriched space
\begin{equation*}
\widetilde{\mathcal{S}}^0_m(\partial K)= \mathcal{S}_m^0(\partial  K)\oplus \bigl\{\widetilde{f}\bigr\},
\end{equation*}
so that
\begin{equation*}
\operatorname{dim}\left(\widetilde{\mathcal{S}}^0_m(\partial K)\right)=mN+1.
\end{equation*}

\begin{remark}
    Using the property~\eqref{ortbl}, we observe that 
    \begin{equation}\label{dsaaa}
        \mathcal{G}\bigl(\widetilde{f}\,\bigr)= \mathcal{G}\left( b_{1}\right)=\sum_{j=1}^{mN}d_j \mathcal{L}_j\left( b_{1}\right)=\sum_{j=1}^{mN}d_j \left\langle b_1, b_j\right\rangle_{\partial K}=\sum_{j=1}^{mN}d_j \delta_{1j}=d_{1}\neq 0. 
    \end{equation}
\end{remark}

In the following, we assume that the functionals $\mathcal{L}_j$, $j = 1, \ldots, mN$, and $\mathcal{F}$ admit extensions to the space $\widetilde{\mathcal{S}}^0_m(\partial K)$, and we write
\begin{equation*}
    \mathcal{L}_j:\widetilde{\mathcal{S}}^0_m(\partial K)\rightarrow\mathbb{R},  \quad \mathcal{F}:\widetilde{\mathcal{S}}^0_m(\partial K)\rightarrow\mathbb{R}, \quad j=1,\dots,mN.
\end{equation*}
Under this assumption, we define the enriched triple
\begin{equation*}
\widetilde{\mathcal{A}}_{m,N}=\bigl(\partial K, \widetilde{\mathcal{S}}^0_m(\partial K), \widetilde{\Sigma}_{\partial K}\bigr),
\end{equation*}
where
\begin{equation*}
\widetilde{\Sigma}_{\partial K}=\left\{\mathcal{L}_j, \mathcal{F}\, :\, j=1,\ldots, mN\right\}.
\end{equation*}

Finally, we establish the following key result.

\begin{theorem}
The triple $\widetilde{\mathcal{A}}_{m,N}$ defines a finite element.
\end{theorem}

\begin{proof}
Let $f \in \widetilde{\mathcal{S}}^0_m(\partial K)$ be such that
\begin{equation}\label{condaa}
\mathcal{L}_j(f) = 0, \quad \mathcal{F}(f) = 0, \quad j = 1,\ldots, mN.
\end{equation}
By definition, we can write
\begin{equation*}
f = p + \alpha \widetilde{f}, \quad p \in  \mathcal{S}_m^0(\partial  K), \quad \alpha \in \mathbb{R}.
\end{equation*}
Since $\mathcal{G}$ is a linear combination of the functionals $\mathcal{L}_j$, $j=1,\ldots,mN$, condition~\eqref{condaa} implies that
\begin{equation*}
\mathcal{G}(f) =\sum_{j=1}^{mN} d_j \mathcal{L}_j(f)=0.
\end{equation*}
On the other hand, since $\mathcal{G}\bigl(\widetilde{f}\bigr) \neq 0$ by~\eqref{dsaaa}, using Lemma~\ref{lemmaslsa} and the linearity of $\mathcal{G}$, we obtain
\begin{equation*}
0 = \mathcal{G}(f) = \mathcal{G}(p) + \alpha \mathcal{G}\bigl(\widetilde{f}\bigr) = \alpha \mathcal{G}\bigl(\widetilde{f}\bigr),
\end{equation*}
which implies that $\alpha = 0$. Therefore,
\begin{equation*}
f = p \in  \mathcal{S}_m^0(\partial  K).
\end{equation*}
By Lemma~\ref{fffffff}, the only function $f\in  \mathcal{S}_m^0(\partial  K)$ satisfying
\begin{equation*}
    \mathcal{L}_j(f)=0, \quad \mathcal{F}(f)=0, \quad j=1,\dots,mN,
\end{equation*}
is $f=0$.  This proves the unisolvence of the triple $\widetilde{\mathcal{A}}_{m,N}$. The proof is completed.  
\end{proof}

\section{Families of nonconforming finite elements}\label{sec3}

Let $K \subset \mathbb{R}^2$ be a polygon with $N$ edges. We denote by
\begin{equation*}
    \left(K, \mathbb{V}_m(K), \Sigma_K \right)
\end{equation*}
the standard Lagrange finite element of degree $m \in \mathbb{N}$, where $\mathbb{V}_m(K)$ denotes a suitable polynomial space on $K$, depending on its geometry, and 
\begin{equation*}
\Sigma_K=\left\{\mathcal{I}_r \, :\, r=1,\ldots,R\right\},    
\end{equation*}
is the associated set of degrees of freedom~\cite{Guessab:2022:SAB}. We define the subspace of functions vanishing on the boundary as
\begin{equation*}
    \mathbb{V}_{m,0}(K)= \mathbb{V}_m(K) \cap H_0^1(K).
\end{equation*}
Since $\mathbb{V}_{m,0}(K)$ has lower dimension, there exists a proper subset 
\begin{equation*}
\Sigma_{K,0}=\left\{\mathcal{I}_s^{\prime} \, :\, s=1,\ldots,S\right\} \subset \Sigma_K, \quad S<R,    
\end{equation*}
such that the triple
\begin{equation}\label{aus211}
    \left(K, \mathbb{V}_{m,0}(K), \Sigma_{K,0} \right)
\end{equation}
defines a finite element.
\smallskip

To extend the finite element $\mathcal{A}_{m,N}$, defined on the boundary $\partial K$, to the full domain $K$, we distinguish two cases:
\begin{itemize}
    \item \textit{If $\mathcal{A}_{m,N}$ defines a finite element}, then we set
    \begin{equation*}
        \widehat{\mathbb{V}}_m(K) = \mathbb{V}_{m,0}(K)\oplus  \mathcal{S}_m^0(\partial  K), \quad \widehat{\Sigma}_K = \Sigma_{K,0} \cup \Sigma_{\partial K},
    \end{equation*}
    where $\Sigma_{\partial K}$ is defined in~\eqref{partialK}.
    \vspace*{2mm}

    \item \textit{If $\mathcal{A}_{m,N}$ does not define a finite element}, then we set
    \begin{equation*}
        \widehat{\mathbb{V}}_m(K) = \mathbb{V}_{m,0}(K)\oplus  \mathcal{S}_m^0(\partial  K)\oplus \bigl\{\widetilde{f}\bigr\}, \quad  \widehat{\Sigma}_K = \Sigma_{K,0} \cup \Sigma_{\partial K} \cup \{ \mathcal{F} \},
    \end{equation*}
    where $\mathcal{F}$ is the functional defined in~\eqref{deff} and $\widetilde{f}$ is the enrichment function defined in~\eqref{fhat}.
\end{itemize}

In both cases, we define
\begin{equation*}
\widehat{\mathcal{A}}_{m,N}=\bigl(K, \widehat{\mathbb{V}}_m(K), \widehat{\Sigma}_K\bigr).
\end{equation*}
Finally, the following important theorem can be proved.

\begin{theorem}
The triple $\widehat{\mathcal{A}}_{m,N}$ defines a finite element.
\end{theorem}
\begin{proof}
We prove the result under the assumption that $\mathcal{A}_{m,N}$ defines a finite element; the other case can be treated analogously. 

Let $f \in \widehat{\mathbb{V}}_m(K)$ be such that
\begin{align*}
\mathcal{L}_j(f) &= 0, \quad \mathcal{L}_j\in\Sigma_{\partial K}, \quad  j = 1, \dots, mN, \\[2mm]
\mathcal{I}^{\prime}_s(f)&=0,  \quad \mathcal{I}^{\prime}_s\in\Sigma_{K,0}, \quad s=1,\dots,S.
\end{align*}
By definition of $\widehat{\mathbb{V}}_m(K)$, there exist $g_0 \in \mathbb{V}_{m,0}(K)$ and $p \in  \mathcal{S}_m^0(\partial  K)$ such that
\begin{equation*}
f = g_0+p.
\end{equation*}
Since $g_0$ vanishes on $\partial K$, by using the linear functionals of $\Sigma_{\partial K}$, we have
\begin{equation*}
0=\mathcal{L}_j(f)=\mathcal{L}_j\left(g_0\right)+\mathcal{L}_j(p)=\mathcal{L}_j(p), \quad j=1,\dots,mN.
\end{equation*}
As $\mathcal{A}_{m,N}$ defines a finite element, these conditions imply that $p = 0$. Thus 
\begin{equation*}
    f = g_0.
\end{equation*} 
 
Applying now the conditions on the interior degrees of freedom, we obtain
\begin{equation*}
0=\mathcal{I}^{\prime}_s(f)=\mathcal{I}_s^{\prime}\left(g_0\right), \quad s=1,\dots, S.
\end{equation*}
Since~\eqref{aus211} defines a finite element, it results that $g_0 = 0$, and therefore $f = 0$. 
This proves that 
\begin{equation*}
   \widehat{\mathcal{A}}_{m,N}=\bigl(K, \widehat{\mathbb{V}}_m(K), \widehat{\Sigma}_K\bigr)
\end{equation*}
defines a finite element over $K$.  
\end{proof}

\section{Numerical results}
\label{sec4}

In this section, we present numerical experiments aimed at assessing the performance and accuracy of the proposed nonconforming approximation methods. All tests are performed on the square domain $\Omega = [-1,1]^2$, using a set of benchmark functions that exhibit different analytic behaviors. The test functions considered are:
 \begin{eqnarray*}
    f_1(x,y)&=&\sqrt{x^2+y^2}, \ \ f_2(x,y)=e^{-4\left(x^2+y^2\right)}\sin\left(\pi(x+y)\right),\\
    f_3(x,y)&=&\sin(2\pi x)\sin(2\pi y),\ \, 
    f_4(x,y)=\sin\left(4\pi(x+y)\right),\ \,
    f_5(x,y)=\frac{1}{25(x^2+y^2)+1},\\ 
 f_6(x,y)&=&0.75\exp\biggl(-\frac{(9(x+1)/2-2)^2}{4}-\frac{(9(y+1)/2-2)^2}{4}\biggr)\\[-2pt]
 &&+0.75\exp\biggl(-\frac{(9(x+1)/2+1)^2}{49}-\frac{(9(y+1)/2+1)}{10}\biggr)\\[-2pt]
	&&+ 0.5\exp\biggl(-\frac{(9(x+1)/2-7)^2}{4}-\frac{(9(y+1)/2-3)^2}{4}\biggr)\\
	&&-0\widehat.2\exp\biggl(-(9(x+1)/2-4)^2-(9(y+1)/2-7)^2\biggr).
\end{eqnarray*}
The function $f_6$ is the well-known Franke function, widely used as a benchmark for approximation methods~\cite{franke1982scattered}.

  The proposed nonconforming approximation schemes are implemented on both triangular and quadrilateral meshes.  All experiments are carried out using the \texttt{Matlab} software. The degrees of freedom are computed using high-order Gaussian quadrature rules to ensure adequate numerical precision. For quadrilateral elements, the approximation spaces are constructed using scaled Legendre polynomial bases, generated via the \texttt{legpoly} routine from the Chebfun package~\cite{hale2014fast}.

\subsection{Nonconforming approximation methods using triangular finite elements}\label{ssec5.1}

We start by considering the family of regular Friedrichs--Keller triangulations~\cite{Knabner}, denoted by
$$\mathcal{T}_{n}^{\mathrm{FK}}=\left\{t_i\,:\, i=1,\dots,2(n+1)^2\right\},$$ which consists of $2(n+1)^2$ triangular elements (see Fig.~\ref{triang}).

We compare the behavior of the approximation errors in $L^1$-norm produced by the nonconforming approximation methods $\widehat{\mathcal{A}}_{1,3}$ and $\widehat{\mathcal{A}}_{2,3}$ across different triangular meshes, both constructed using Jacobi orthogonal polynomials on $[-1,1]$ with parameters $\alpha = 3$ and $\beta = 2$. The results, shown in Fig.~\ref{triangL11}, demonstrate the expected error decay, with $\widehat{\mathcal{A}}_{2,3}$ exhibiting superior accuracy on finer meshes.
\begin{figure}[htb]
    \centering
    \includegraphics[width=\linewidth]{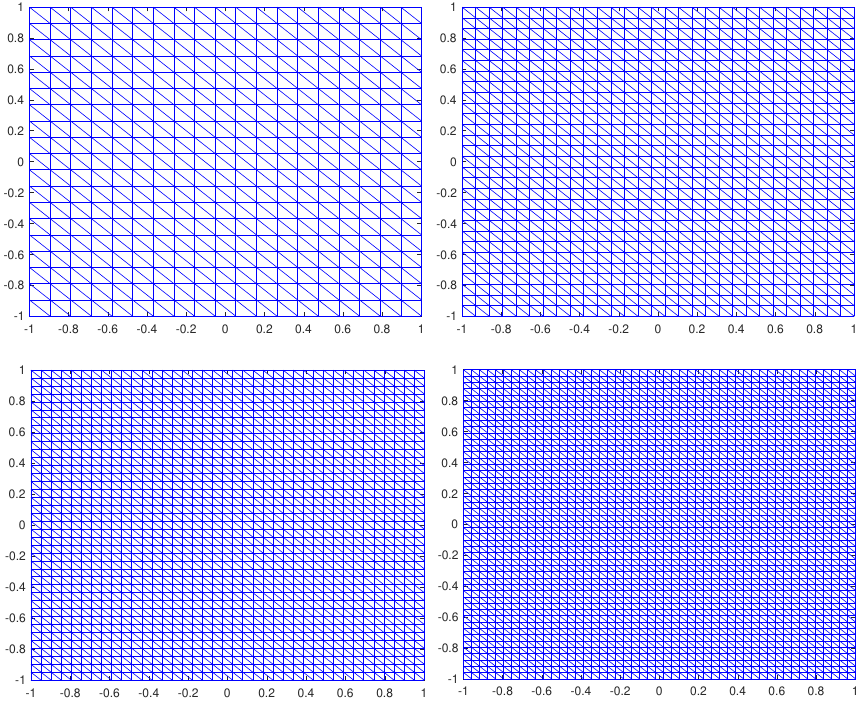}
  \caption{Regular Friedrichs--Keller triangulations $\mathcal{T}_n^{\mathrm{FK}}$, for $n = 20, 30, 40, 50$.}\label{triang}
\end{figure}

Next, we examine the influence of different Jacobi parameters. To this end, in Table~\ref{tab1}, we report the errors in $L^1$-norm produced by the methods $\widehat{\mathcal{A}}_{1,3}$ and $\widehat{\mathcal{A}}_{2,3}$ on the triangulation $\mathcal{T}_{30}^{\mathrm{FK}}$, using different parameters $\alpha,\beta$. (Numbers in parentheses indicate decimal exponents.)
The results indicate that while the errors are comparable across the tested values, no single pair $(\alpha,\beta)$ performs optimally for all test functions. This suggests that, in practical applications, the choice of Jacobi parameters may be tailored to the specific features of the target function.
\begin{figure}
    \centering
    \includegraphics[width=0.95\linewidth]{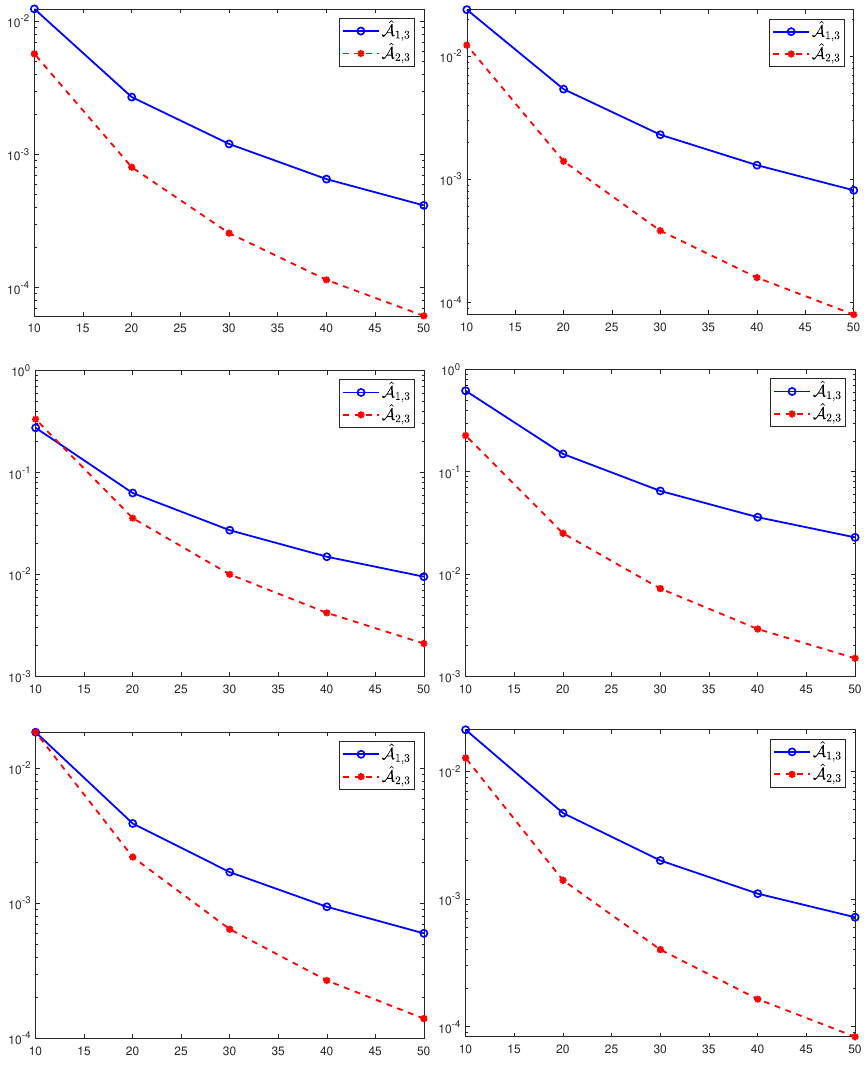}
    \caption{Convergence of the approximation error in $L^1$-norm for the functions $f_1$ -- $f_6$ produced by the nonconforming methods  $\widehat{\mathcal{A}}_{1,3}$ (blue line) and $\widehat{\mathcal{A}}_{2,3}$ (red dashed line), obtained using the Jacobi orthogonal polynomials with parameters $\alpha=3$ and $\beta=2$. The errors are computed over Delaunay triangulations with an increasing number of triangles. The plots are ordered  as follows:  
$f_1$ (top left), $f_2$ (top right),  
$f_3$ (middle left), $f_4$ (middle right),  
$f_5$ (bottom left), and $f_6$ (bottom right).}
    \label{triangL11}
\end{figure}

To further illustrate the behavior of the methods, in Fig.~\ref{f1lin} we present visual reconstructions of the test functions $f_1$ -- $f_6$, produced by $\widehat{\mathcal{A}}_{1,3}$ and $\widehat{\mathcal{A}}_{2,3}$ on the triangulation $\mathcal{T}_{30}^{\mathrm{FK}}$, using Jacobi polynomials with $\alpha = 3$ and $\beta = 2$. These values are not chosen to be \emph{optimal} in terms of minimizing the approximation error. Rather, numerical experiments show that the error initially depends on the choice of Jacobi parameters, but tends to stabilize once the parameters become sufficiently large. Therefore, the values $\alpha = 3$ and $\beta = 2$ are used as representative parameters beyond which the approximation quality remains essentially unchanged.
\begin{table}[h!] 
\caption{Approximation error in $L^1$-norm produced by the nonconforming methods  $\widehat{\mathcal{A}}_{1,3}$ and $\widehat{\mathcal{A}}_{2,3}$ on the triangulation $\mathcal{T}^{\mathrm{FK}}_{30}$, obtained using the Jacobi orthogonal polynomials on $[-1,1]$ with different parameters $\alpha,\beta$.}
 {\tabcolsep 2pt\renewcommand{\arraystretch}{1.1}
\begin{tabular}{c|ccccc|ccccc}
\hline
\multicolumn{1}{c|}{} & \multicolumn{5}{c|}{\(\widehat{\mathcal{A}}_{1,3}\)} & \multicolumn{5}{c}{\(\widehat{\mathcal{A}}_{2,3}\)} \\
\cline{2-6}\cline{7-11} 
$(\alpha,\beta)$ & $(1/2,1)$ & $(1,3/2)$ & $(2,1)$ & $(0,2)$ & $(3,2)$ & $(1/2,1)$ & $(1,3/2)$ & $(2,1)$ & $(0,2)$ & $(3,2)$ \\
\hline
$f_1$ & $1.2(-3)$ & $1.2(-3)$ & $1.2(-3)$ & $1.7(-3)$ & $1.2(-3)$ & $2.7(-4)$  & $3.2(-4)$ & $2.0(-4)$  & $1.1(-4)$ & $2.5(-4)$ \\
$f_2$ & $2.3(-3)$ & $2.4(-3)$  & $2.4(-3)$ & $3.6(-3)$  & $2.3(-3)$ & $4.3(-4)$ & $5.1(-4)$ & $3.0(-4)$ & $1.9(-4)$ & $3.8(-4)$ \\
$f_3$ & $2.8(-2)$ & $2.7(-2)$ & $2.8(-2)$ & $4.1(-2)$  & $2.7(-2)$ & $1.2(-2)$ & $1.4(-2)$ & $7.6(-3)$  & $4.0(-3)$ & $9.9(-3)$\\
$f_4$ & $6.8(-2)$ & $6.5(-2)$ & $6.9(-2)$  & $1.0(-1)$ & $6.5(-2)$ & $7.6(-3)$ & $7.3(-3)$ & $7.5(-3)$ & $8.2(-3)$ & $7.2(-3)$ \\
$f_5$ & $1.6(-3)$ & $1.6(-3)$ & $1.7(-3)$ & $2.6(-3)$ & $1.7(-3)$ & $7.3(-4)$ & $8.8(-4)$ & $5.0(-4)$ & $2.8(-4)$ & $6.4(-4)$\\
$f_6$ & $2.1(-3)$ & $2.0(-3)$ & $2.1(-3)$ & $3.2(-3)$ & $2.0(-3)$ & $4.5(-4)$ & $5.4(-4)$ & $3.1(-4)$ & $1.9(-4)$ & $4.0(-4)$  \\
\hline
\end{tabular}}
\label{tab1}
\end{table}

\begin{figure}
    \centering
    \includegraphics[width=0.85\linewidth]{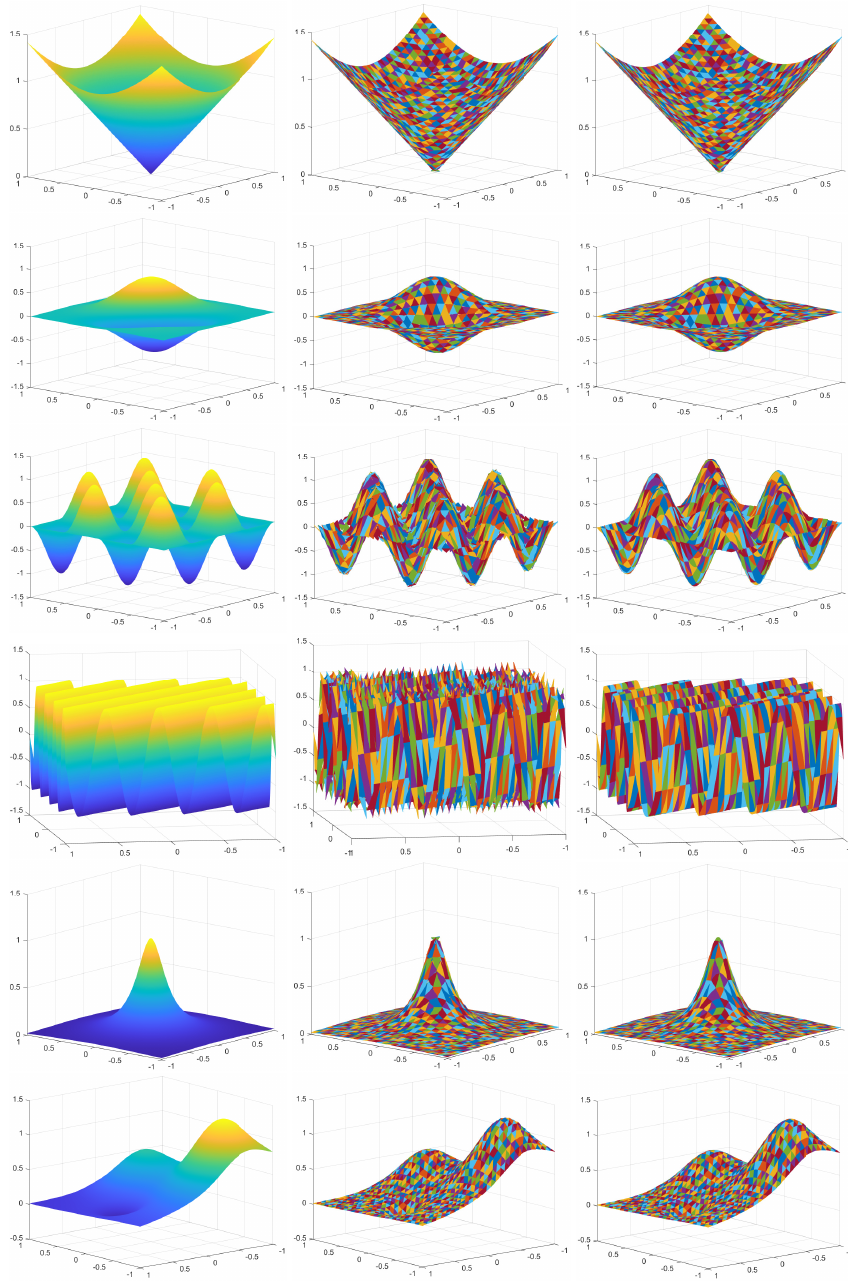}
    \caption{From top to bottom: plots of the functions $f_1, \ldots, f_6$ (left), and their reconstructions using the nonconforming methods  $\widehat{\mathcal{A}}_{1,3}$ (center) and $\widehat{\mathcal{A}}_{2,3}$ (right).}
    \label{f1lin}
\end{figure}

\subsection{Nonconforming approximation methods using quadrilateral finite elements}

We now consider uniform Cartesian partitions 
$\mathcal{Q}_{n}=\left\{s_i\,:\, i=1,\ldots,(n-1)^2\right\}$
each consisting of $(n-1)^2$ square elements; see Fig.~\ref{quad}. 
\begin{figure}
    \centering
    \includegraphics[width=\linewidth]{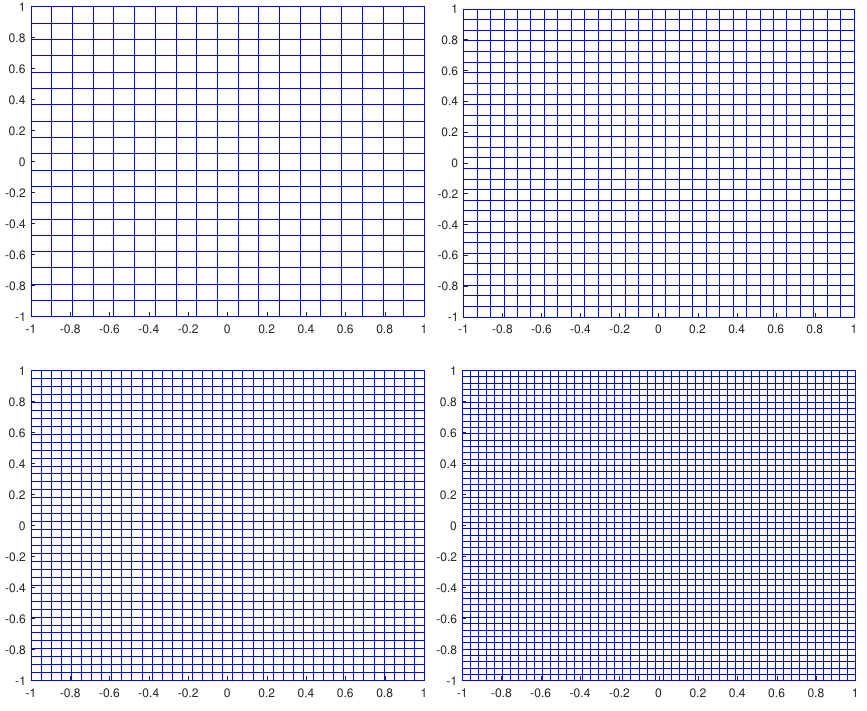}
    \caption{Successive uniform Cartesian partitions $\mathcal{Q}_{n}$, with $n=20,30,40,50$.}
    \label{quad}
\end{figure}
\begin{table}[h!] 
\caption{Approximation error in $L^1$-norm produced by the nonconforming methods  $\widehat{\mathcal{A}}_{1,4}$ and $\widehat{\mathcal{A}}_{2,4}$ on the partition $\mathcal{Q}_{30}$, obtained using the Jacobi orthogonal polynomials on $[-1,1]$ with different parameters $\alpha,\beta$.}
 {\tabcolsep 2pt\renewcommand{\arraystretch}{1.1}
\begin{tabular}{c|ccccc|ccccc}
\hline
\multicolumn{1}{c|}{} & \multicolumn{5}{c|}{\(\widehat{\mathcal{A}}_{1,4}\)} & \multicolumn{5}{c}{\(\widehat{\mathcal{A}}_{2,4}\)} \\
\cline{2-6}\cline{7-11} 
$(\alpha,\beta)$ & $(1/2,1)$ & $(1,3/2)$ & $(2,1)$ & $(0,2)$ & $(3,2)$ & $(1/2,1)$ & $(1,3/2)$ & $(2,1)$ & $(0,2)$ & $(3,2)$ \\
\hline
$f_1$ & $2.4(-3)$ & $2.9(-3)$ & $2.0(-3)$ & $1.7(-3)$ & $2.5(-3)$ & $5.3(-5)$  & $5.3(-5)$ & $5.1(-5)$  & $5.2(-5)$ & $5.1(-5)$ \\
$f_2$ & $6.6(-3)$ & $7.9(-3)$  & $5.5(-3)$ & $4.5(-3)$  & $6.8(-3)$ & $1.5(-4)$ & $1.5(-4)$ & $1.5(-4)$ & $1.5(-4)$ & $1.4(-4)$ \\
$f_3$ & $2.3(-2)$ & $2.2(-2)$ & $2.2(-2)$ & $2.8(-2)$  & $2.1(-2)$ & $1.0(-3)$ & $1.0(-3)$ & $9.8(-4)$ & $1.0(-3)$ & $9.4(-4)$\\
$f_4$ & $1.4(-1)$ & $1.4(-1)$ & $1.4(-1)$  & $1.8(-1)$ & $1.3(-1)$ & $1.0(-2)$ & $1.0(-2)$ & $9.9(-3)$ & $1.0(-2)$ & $9.6(-3)$ \\
$f_5$ & $3.5(-3)$ & $4.5(-3)$ & $2.9(-3)$ & $1.9(-3)$ & $3.8(-3)$ & $1.1(-4)$ & $1.0(-4)$ & $1.0(-4)$ & $1.0(-4)$ & $9.8(-5)$\\
$f_6$ & $4.4(-3)$ & $5.3(-3)$ & $3.6(-3)$ & $2.7(-3)$ & $4.6(-3)$ & $9.5(-5)$ & $9.2(-5)$ & $8.9(-5)$ & $9.3(-5)$ & $8.6(-5)$\\
\hline
\end{tabular}}
\label{tab2}
\end{table}

\begin{figure}
    \centering
    \includegraphics[width=0.95\linewidth]{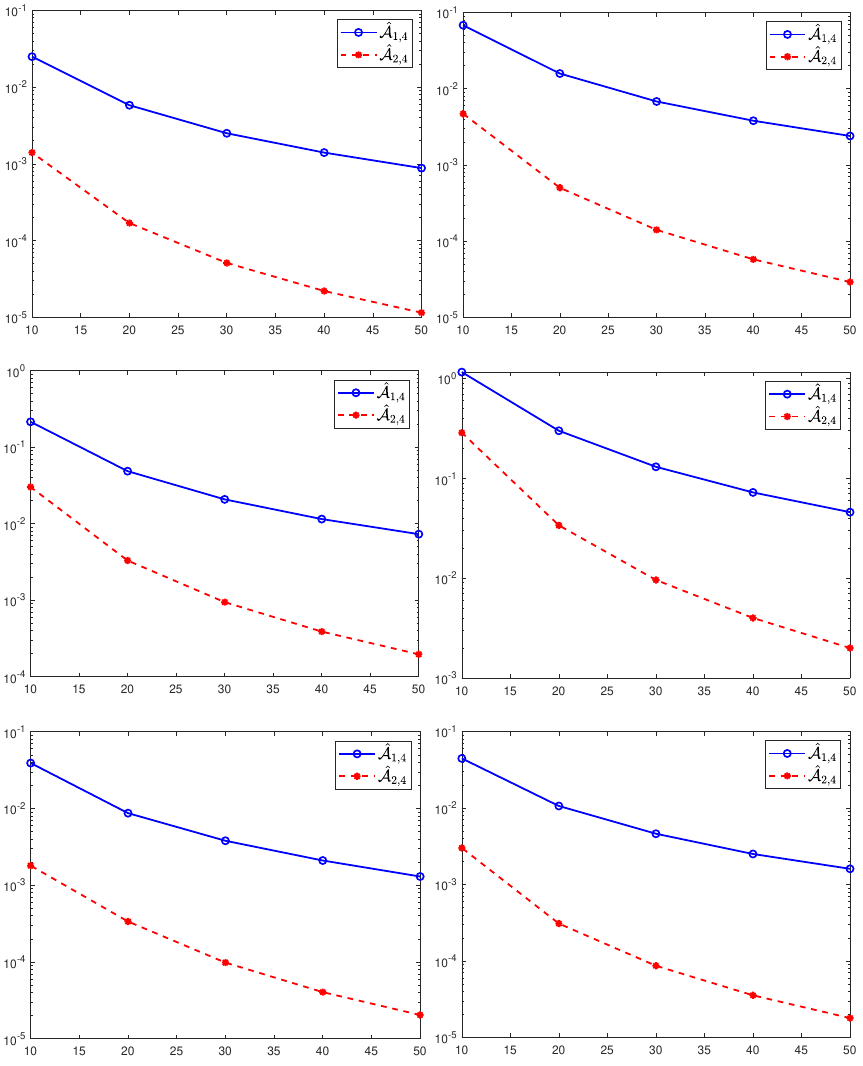}
    \caption{Convergence of the approximation error in $L^1$-norm for the functions $f_1$ -- $f_6$ produced by the nonconforming methods  $\widehat{\mathcal{A}}_{1,4}$ (blue line) and $\widehat{\mathcal{A}}_{2,4}$ (red dashed line) obtained using the Jacobi orthogonal polynomials with parameters $\alpha=3$ and $\beta=2$. The errors are computed over Cartesian partitions with an increasing number of squares. The plots are ordered  as follows:  
$f_1$ (top left), $f_2$ (top right),  
$f_3$ (middle left), $f_4$ (middle right),  
$f_5$ (bottom left), and $f_6$ (bottom right).}
    \label{quad1}
\end{figure}

\begin{figure}
    \centering
    \includegraphics[width=0.85\linewidth]{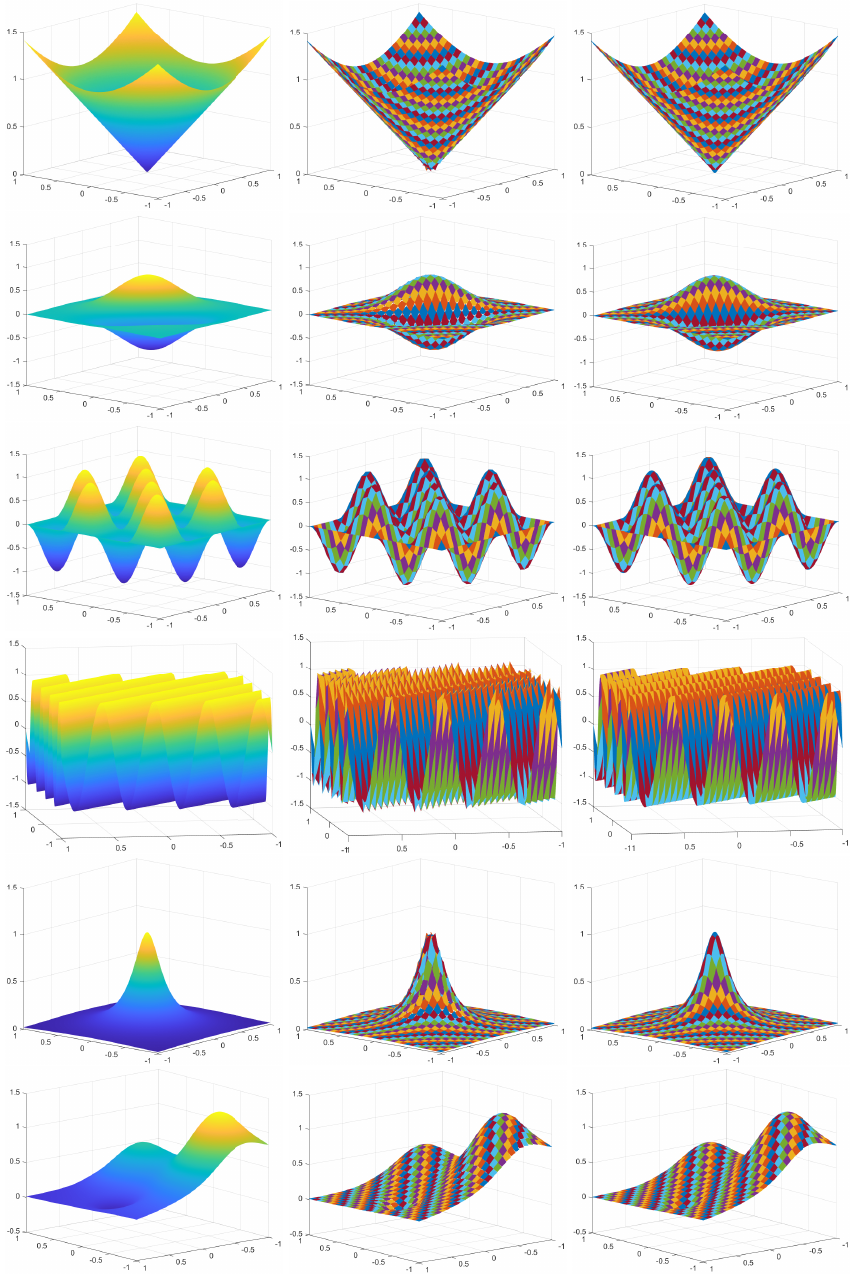}
    \caption{From top to bottom: plots of the functions $f_1, \ldots, f_6$ (left), and their reconstructions using the nonconforming methods  $\widehat{\mathcal{A}}_{1,4}$ (center) and $\widehat{\mathcal{A}}_{2,4}$ (right).}
    \label{f1lins}
\end{figure}

We compare the performance of the nonconforming approximation methods $\widehat{\mathcal{A}}_{1,4}$ and $\widehat{\mathcal{A}}_{2,4}$, both constructed using Jacobi orthogonal polynomials on  $[-1,1]$ with parameters $\alpha = 3$ and $\beta = 2$. As shown in Fig.~\ref{quad1}, both methods exhibit a decreasing trend of the error in $L^1$-norm as the mesh is refined. In analogy with the triangular case, the enriched method $\widehat{\mathcal{A}}_{2,4}$ demonstrates a faster convergence rate. Moreover, in the quadrilateral setting, the difference in accuracy between the two methods becomes more significant with finer meshes.

Next, we examine the influence of different Jacobi parameters. To this end, in Table~\ref{tab2}, we compute the errors in $L^1$-norm 
produced by $\widehat{\mathcal{A}}_{1,4}$ and $\widehat{\mathcal{A}}_{2,4}$ on the partition $\mathcal{Q}_{30}$, using different parameters $\alpha,\beta$. The behavior closely mirrors that observed in the triangular case: the optimal choice of parameters may depend on the specific properties of the target function. 

Finally, to better highlight the efficiency of the proposed method, in Fig.~\ref{f1lins} we show the reconstructions of the functions $f_1$ -- $f_6$ produced by $\widehat{\mathcal{A}}_{1,4}$ and $\widehat{\mathcal{A}}_{2,4}$ on the partition $\mathcal{Q}_{30}$, again using the Jacobi polynomials with parameters $\alpha = 3$ and $\beta = 2$.


\section{Conclusions and Future Work}\label{sec5}
In this paper, we have introduced new families of nonconforming approximation schemes capable of reproducing polynomials of degree~$m$ on general polygonal meshes. The construction is based on degrees of freedom defined through weighted moments of orthogonal polynomials, making the framework particularly suitable for function reconstruction problems where only integral data are available and pointwise values are inaccessible.  A central theoretical contribution is the complete unisolvence analysis, which establishes that the well-posedness of the approximation space depends on the parity of the product between the polynomial degree~$m$ and the number of polygon edges~$N$. In the degenerate cases where unisolvence fails, we have proposed an enrichment strategy involving an additional linear functional and a suitably designed enrichment function. This guarantees that the interpolation scheme remains unisolvent for all combinations of $m$ and $N$. Numerical experiments conducted on representative test cases confirm both the accuracy and robustness of the proposed method. The flexibility of using general orthogonal polynomials and custom weight functions enables the design of application-specific reconstruction schemes, particularly when prior structural information on the data is available. Several families of nonclassical orthogonal polynomials, such as certain sequences of associated Jacobi polynomials, also satisfy the first condition of Theorem~1. Their structural properties and potential benefits in reconstruction contexts will be the subject of a forthcoming study. Future directions include the extension of the framework to three-dimensional polyhedral meshes, the treatment of time-dependent reconstruction problems, and the use of nonclassical or data-adapted weight functions to further enhance accuracy in applications such as signal processing and inverse problems.


\section*{Funding}
This research has been achieved as part of RITA ``Research
 ITalian network on Approximation'' and as part of the UMI group ``Teoria dell'Approssimazione
 e Applicazioni''. The research was supported by GNCS-INdAM 2025 project ``Polinomi, Splines e Funzioni Kernel: dall'Approssimazione Numerica al Software Open-Source''. The work of F. Nudo has been funded by the European Union NextGenerationEU under the Italian National Recovery and Resilience Plan (PNRR), Mission 4, Component 2, Investment 1.2 ``Finanziamento di progetti presentati da giovani ricercatori'',\ pursuant to MUR Decree No. 47/2025. The research was supported by the grant  Bando Professori visitatori 2022 which has allowed the visit of Prof. Allal Guessab to the Department of Mathematics and Computer Science of the University of Calabria in the spring 2022. The  research of G.V. Milovanovi\'c   was supported in part by the Serbian Academy of Sciences and Arts, Belgrade (Grant No. $\Phi$-96).

 \section*{Conflict of Interest}
The authors declare that they have no conflict of interest.


\end{document}